\documentclass[reqno]{amsart}

\usepackage{amssymb,mathrsfs}
\usepackage{graphicx,xcolor}
\usepackage[bookmarks,hyperfootnotes=false]{hyperref}
\usepackage[nameinlink,capitalise,noabbrev]{cleveref}
\usepackage[msc-links]{amsrefs} 
\usepackage{doi}
   \def\lowfwd #1#2#3{{\mathop{\kern0pt #1}\limits^{\kern#2pt\raise.#3ex
\vbox to 0pt{\hbox{$\scriptscriptstyle\rightarrow$}\vss}}}}
\def\lowbkwd #1#2#3{{\mathop{\kern0pt #1}\limits^{\kern#2pt\raise.#3ex
\vbox to 0pt{\hbox{$\scriptscriptstyle\leftarrow$}\vss}}}}
\def\lowfbwd #1#2#3{{\mathop{\kern0pt #1}\limits^{\kern#2pt\raise.#3ex
\vbox to 0pt{\hbox{$\scriptscriptstyle\leftrightarrow$}\vss}}}}
\def\fwd #1#2{{\lowfwd{#1}{#2}{15}}}
\def\ve{\kern-1.5pt\lowfwd e{1.5}2\kern-1pt}
\def\vedash{{\mathop{\kern0pt e\lower.5pt\hbox{${}
     \scriptstyle'$}}\limits^{\kern0pt\raise.02ex
     \vbox to 0pt{\hbox{$\scriptscriptstyle\rightarrow$}\vss}}}}
\def\ev{\kern-1pt\lowbkwd e{0.5}2\kern-1pt}
\def\vE{{\hskip-1pt{\fwd E3}\hskip-1pt}} 
\def\vf{\kern-2pt\lowfwd f{2.5}2\kern-1pt}
\def\vF{{\hskip-1pt{\fwd F3}\hskip-1pt}} 
\def\vfdash{{\mathop{\kern0pt f\raise 1pt\hbox{${}
     \scriptstyle'$}}\limits^{\kern2pt\raise.02ex
     \vbox to 0pt{\hbox{$\scriptscriptstyle\rightarrow$}\vss}}}}
\def\fv{\lowbkwd f01}

\def\vr{\lowfwd r{1.5}2}
\def\rv{\lowbkwd r02}
\def\vro{\lowfwd {r_0}12}

\def\vrdash{{\mathop{\kern0pt r\lower.5pt\hbox{${}
     \scriptstyle'$}}\limits^{\kern0pt\raise.02ex
     \vbox to 0pt{\hbox{$\scriptscriptstyle\rightarrow$}\vss}}}}
\def\rvdash{{\mathop{\kern0pt r\lower.5pt\hbox{${}
     \scriptstyle'$}}\limits^{\kern0pt\raise.02ex
     \vbox to 0pt{\hbox{$\scriptscriptstyle\leftarrow$}\vss}}}}
\def\vrdashp{{\mathop{\kern0pt r_p\kern-4pt\lower.5pt\hbox{${}
     \scriptstyle'$}}\limits^{\kern0pt\raise.02ex
     \vbox to 0pt{\hbox{$\scriptscriptstyle\rightarrow$}\vss}}}\,}
\def\rvdashp{{\mathop{\kern0pt r_p\kern-4pt\lower.5pt\hbox{${}
     \scriptstyle'$}}\limits^{\kern0pt\raise.02ex
     \vbox to 0pt{\hbox{$\scriptscriptstyle\leftarrow$}\vss}}}\,}
\def\vrddash{{\mathop{\kern0pt r\lower.5pt\hbox{${}
     \scriptstyle''$}}\limits^{\kern0pt\raise.02ex
     \vbox to 0pt{\hbox{$\scriptscriptstyle\rightarrow$}\vss}}}}

\def\vrm{\lowfwd {r_m}{-1}1}

\def\vrn{\lowfwd {r_n}11}

\def\vs{\hskip-1pt\lowfwd s{1.5}1}
\def\sv{{{\hskip-1pt\lowbkwd s{1}1}\hskip-1pt}}
\def\vsdash{{\mathop{\kern0pt s\lower.5pt\hbox{${}
     \scriptstyle'$}}\limits^{\kern0pt\raise.02ex
     \vbox to 0pt{\hbox{$\scriptscriptstyle\rightarrow$}\vss}}}}
\def\svdash{{\mathop{\kern0pt s\lower.5pt\hbox{${}
     \scriptstyle'$}}\limits^{\kern0pt\raise.02ex
     \vbox to 0pt{\hbox{$\scriptscriptstyle\leftarrow$}\vss}}}}
\def\vsddash{{\mathop{\kern0pt s\lower.5pt\hbox{${}
     \scriptstyle''$}}\limits^{\kern0pt\raise.02ex
     \vbox to 0pt{\hbox{$\scriptscriptstyle\rightarrow$}\vss}}}}
\def\svddash{{\mathop{\kern0pt s\lower.5pt\hbox{${}
     \scriptstyle''$}}\limits^{\kern0pt\raise.02ex
     \vbox to 0pt{\hbox{$\scriptscriptstyle\leftarrow$}\vss}}}}
\def\vsdashp{{\mathop{\kern0pt s_p\kern-4pt\lower.5pt\hbox{${}
     \scriptstyle'$}}\limits^{\kern0pt\raise.02ex
     \vbox to 0pt{\hbox{$\scriptscriptstyle\rightarrow$}\vss}}}\,}
\def\svdashp{{\mathop{\kern0pt s_p\kern-4pt\lower.5pt\hbox{${}
     \scriptstyle'$}}\limits^{\kern0pt\raise.02ex
     \vbox to 0pt{\hbox{$\scriptscriptstyle\leftarrow$}\vss}}}\,}
\def\vso{\lowfwd {s_0}11}

\def\vsone{\lowfwd {s_1}11}

\def\vsidash{{\mathop{\kern0pt s_i\kern-3.5pt\lower.3pt\hbox{${}
     \scriptstyle'$}}\limits^{\kern0pt\raise.02ex
     \vbox to 0pt{\hbox{$\scriptscriptstyle\rightarrow$}\vss}}}}

\def\vsm{\lowfwd {s_m}{-1}1}

\def\vsn{\lowfwd {s_n}11}

\def\vsnplusone{\lowfwd {s_{n+1}}{-9}1}

\def\vsqdash{{\mathop{\kern0pt s_q\kern-3.5pt\lower.3pt\hbox{${}
     \scriptstyle'$}}\limits^{\kern0pt\raise.02ex
     \vbox to 0pt{\hbox{$\scriptscriptstyle\rightarrow$}\vss}}}}

\def\vtdash{{\mathop{\kern0pt t\lower-.5pt\hbox{${}
     \scriptstyle'$}}\limits^{\kern0pt\raise.1ex
     \vbox to 0pt{\hbox{$\scriptscriptstyle\rightarrow$}\vss}}}}
\def\tvdash{{\mathop{\kern0pt t\lower-.5pt\hbox{${}
     \scriptstyle'$}}\limits^{\kern0pt\raise.1ex
     \vbox to 0pt{\hbox{$\scriptscriptstyle\leftarrow$}\vss}}}}
\def\vsl{\lowfwd {s_\ell}11}
\def\svl{\lowbkwd {s_\ell}{-1}2}

\def\vsv{\lowfwd {s_v}11}
\def\svv{\lowbkwd {s_v}{-1}2}

\def\vS{{\hskip-1pt{\fwd S3}\hskip-1pt}} 
\def\vSi{\lowfwd {S_i}11}
\def\vSk{\lowfwd {S_k}11}
\def\vSk{\lowfwd {S_k}11}

\def\vSstar{{\mathop{\kern0pt S\lower-1pt\hbox{$^*$}}\limits^{\kern2pt
     \vbox to 0pt{\hbox{$\scriptscriptstyle\rightarrow$}\vss}}}}
\def\vSdash{{\mathop{\kern0pt S\lower-1pt\hbox{${}
     \scriptstyle'$}}\limits^{\kern2pt\raise.1ex
     \vbox to 0pt{\hbox{$\scriptscriptstyle\rightarrow$}\vss}}}}

\def\vt{\lowfwd t{1.5}1}
\def\tv{\lowbkwd t{1.5}1}

\def\vU{{\vec U}} 
\def\vUk{\lowfwd {U\!_k}31}

\usepackage{enumitem}
\usepackage{comment}

\colorlet{darkishRed}{red!80!black}
\colorlet{darkishBlue}{blue!60!black}
\colorlet{darkishGreen}{green!60!black}
\hypersetup{
  draft=false,
  bookmarksopen=true,
  colorlinks,
  linkcolor=darkishRed,
  citecolor=darkishGreen,
  urlcolor=darkishBlue
}

\renewcommand{\PrintDOI}[1]{\doi{#1}}

\renewcommand{\leq}{\leqslant}
\renewcommand{\geq}{\geqslant}

\newcommand{\R}{\mathbb{R}}
\newcommand{\N}{\mathbb{N}}
\newcommand{\F}{\mathcal{F}}

\newcommand{\RR}{\mathcal{R}}

\newtheorem{theorem}{Theorem}[section] 
\newtheorem{corollary}[theorem]{Corollary}
\newtheorem{lemma}[theorem]{Lemma}

\theoremstyle{definition}
\newtheorem{definition}[theorem]{Definition}

\setlist[enumerate]{
  label=\textrm{(\roman*)},   
  font=\normalfont            
}
\def\COMMENT{}
\def\ucl(#1){\lfloor #1 \rfloor}
\def\td{tree-decom\-pos\-ition}

\title{Tangle structure trees~II:\\ trees of tangles and tangle-tree duality\hbox to 3mm{\hfil}}
\author{Hanno von Bergen and Reinhard Diestel}
\date{\today}

\begin{document}

\maketitle

\begin{abstract}
Tangle structure trees, introduced in~\cite{TSTs}, offer a unified data structure that displays all the tangles of a graph or data set together with certificates for the non-existence of any other tangles, either locally%
   \COMMENT{}
   or overall. In this paper we apply tangle structure trees to derive new versions of the two fundamental tangle theorems: the tree-of-tangles theorem, and the tangle-tree duality theorem.

We extend the tree-of-tangles theorem to $\F$-tangles that need not be profiles. When $\F$ consists of stars of separations, as it does in classical tangle-tree duality theorems, we show how to convert tangle structure trees that certify the non-existence of $\F$-tangles into \td s that certify this in the way known from graph tangles, as $S$-trees over~$\F$.
\end{abstract}

\medskip\section{Introduction}

\noindent
   In this sequel to~\cite{TSTs} we explore further applications of {\em tangle structure trees\/}. These are data structures, introduced in~\cite{TSTs}, which simultaneously display

\begin{itemize}\itemsep=2pt
\item all the tangles of a given graph or abstract separation system; and

\item for all its orientations%
   \COMMENT{}
   that are not tangles, certificates of why they are not.  
\end{itemize}

\noindent
   Previously, there were two separate types of theorem in tangle theory for these two purposes: {\em tree-of-tangle\/} theorems, which display all the tangles of a graph or data set in a tree-like way that `distinguishes' different tangles by specifying a minimum-order separation that they orient differently; and {\em tangle-tree duality\/} theorems that display, also in a tree-like way, certificates for why the various orientations%
   \COMMENT{}
   of the separation system are not tangles~-- typically in the form of small subsets of oriented separations that are {\em forbidden\/} in the given type of tangle. These two fundamental theorems, which exist in various forms and degrees of generality, are thought of as the two pillars of both graph and abstract tangle theory.

The unified way in which tangle structure trees serve both these purposes still allows us to extract, if desired, separate corollaries of the two types. In~\cite{TSTs} we did this for the second type: we derived new tangle-tree duality-type theorems for tangles whose defining forbidden subsets were not `stars' of separations,%
   \COMMENT{}
   nested sets of typically three oriented separations pointing towards each other. All previously known tangle-tree duality theorems, in particular those for graphs, were for tangles defined by excluding such stars.

Our first main result in this paper is to extract from the tangle structure trees established in~\cite{TSTs} a new theorem of the first type: a~tree-of-tangles theorem for tangles whose set~$\F$ of defining forbidden subsets need not form a `profile'. All known tree-of-tangles theorems so far are for profiles, in fact, for {\em robust\/} profiles. These are still a sweeping generalization of graph tangles. But we shall prove here that we can generalize them further: we shall obtain a tree-of-tangles theorem for all $\F$-tangles whose set~$\F$ of forbidden subsets of oriented separations is robust, even if not a profile.%
   \COMMENT{}

Our second main result is about $\F$-tangles whose set~$\F$ of forbidden subsets consists of stars of separations, as in all the classical tangle-tree duality theorems. We prove a general conversion theorem, which allows us to extract from our tangle structure tree for $\F$-tangles, of a separation system~$\vS$, say, an $S$-tree over~$\F$ if $S$ admits no $\F$-tangles. Such an `$S$-tree over~$\F$' is the tree structure in the most general tangle-tree duality theorems known so far for abstract separation systems (which include graphs), proved in~\cite{TangleTreeAbstract}. Our conversion theorem allows us to re-prove the main result of~\cite{TangleTreeAbstract} from the premise for the existence of tangle structure trees, which differs from the premise of the tangle-tree duality theorem in~\cite{TangleTreeAbstract}.

However we demonstrate that from our premise for the existence of tangle structure trees we can derive a weaker version of the premise in~\cite{TangleTreeAbstract}, the version used in practice when the duality theorem from~\cite{TangleTreeAbstract} is applied in~\cite{TangleTreeGraphsMatroids}. As a consequence, we shall be able in~\cite{TSTapps} to use our conversion theorem, coupled with the existence theorems for tangle structure trees in~\cite{TSTs}, to derive all the known applications of the tangle-tree duality theorem in~\cite{TangleTreeAbstract} to concrete structures, such as graphs, matroids, or data sets~\cite{TangleTreeGraphsMatroids}.

The paper is organized as follows. In \cref{sec:basics} we give a summary of tangle basics, just what is technically needed to read this paper. More on graph tangles can be found in~\cite{DiestelBook25}, more on abstract tangles and their applications in~\cite{ASS,TreeSets,TangleBook}. In \cref{sec:TSTs} we introduce tangle structure trees and summarize the main results from~\cite{TSTs}; more on this, including motivational background, can be found there.

\cref{sec:TTD} contains our conversion theorem that extracts from a tangle structure tree, of a separation system~$\vS$ that has no $\F$-tangles, an $S$-tree over~$\F$. \cref{sec:orderfunctions} is an interlude on submodular order functions for separations, which is needed in the  rest of the paper; we show how a non-injective submodular order function can be extended to an injective one without losing submodularity. The tree-of-tangles theorem for $\F$-tangles with arbitrary robust exclusion sets~$\F$ is proved in \cref{sec:ToTs}. In \cref{sec:measure} we show how to recover the essence of the tangle-tree duality theorem of~\cite{TangleTreeAbstract} from our conversion theorem and our existence theorem for tangle structure trees.

\section{Separation systems and their tangles}\label{sec:basics}

\noindent
   In this section we introduce the terminology that abstract tangle theory requires, following~\cite{ASS}.%
   \footnote{There is one difference to~\cite{ASS}: for historical reasons, the partial ordering on~$\vS$ used there is the inverse of ours. So terms like `large' and `small', infima and suprema etc, are reversed.}
    Readers familiar with~\cite{TSTs} may skip this section.

Tangles of graphs are ways of orienting their separations, each towards one of its two sides. Abstract tangles are designed to work in scenarios where there need not be anything to `separate'. In order to retain our intuition from graphs, however, we continue to refer to the things of which our abstract tangles pick one of two variants (which they will indeed do) as `separations'. These are defined by noting some key properties of graph separations and making them into axioms, as follows.

A \emph{separation system} $(\vS, \leq, ^*)$ is a set~$\vS$, whose elements we call \emph{oriented separations}, that comes with a partial ordering~$\leq$ on~$\vS$ and an order-reversing involution $^*\colon \vS \to \vS$. Thus, for any two elements%
   \footnote{We often denote the elements of~$\vS$ by letters with an arrow, in either direction, precisely in order to have a simple way to refer to their dual elements: by reversing the arrow. But the arrow directions have no meaning: an arbitrary element of~$\vS$ could be denoted equally as $\vs$ or as~$\sv$.}
   $\vr,\vs$ of~$\vS$ with $\vr \leq \vs$ we have $\vr{}^* \geq \vs{}^*$. We write $\vs{}^*=:\sv$, and call $\sv$ the \emph{inverse} of $\vs$. While we allow formally that $\vs = \sv$, in which case we call $\vs$ and~$s$ \emph{degenerate}, this does not happen often in practice.%
   \footnote{The only degenerate separation of a graph $G=(V,E)$, for example, is~$\{V,V\}$.}

If a separation system~$\vU$ happens to be a lattice, that is, if there is a supremum $\vr\lor\vs$ and an infimum $\vr\land\vs$ in~$\vU$ for every two elements $\vr,\vs\in\vU$, we call~$\vU$ a {\em universe\/} of separations. It is {\em distributive\/} if it is distributive as a lattice. A~separation system $\vS\subseteq\vU$ is {\em submodular\/} if for every two elements of~$\vS$ either their infimum or their supremum in~$\vU$ also lies in~$\vS$.
 
Very rarely we may have separations $\vs\le\sv$; then $\vs$ is {\em small\/} and $\sv$ is {\em large\/}.%
   \footnote{The small separations of a graph~$G$ are those of the form~$(V,A)$ with $A\subseteq V = V(G)$.}
   Separation systems without small elements are {\em regular\/}.%
   \COMMENT{}
   We say that $\vs$ is \emph{trivial} (and $\sv$ is \emph{co-trivial}) in $\vS$ if there exists a pair of inverse separations $\vr,\rv<\vs$ in~$\vS$. Trivial separations are clearly large, so co-trivial ones are small, but the converse need not hold. See~\cite{ASS} for more on these technicalities if desired.

The set of {\em unoriented separations} in~$(\vS,\le,^*)$ is
$$S:=\{\{\vs,\sv\}:\vs\in\vS\}.$$
We call the elements $\vs,\sv$ of~$s$ its \emph{orientations}. An \emph{orientation of $S$} is a set $\tau\subseteq\vS$ that contains exactly one orientation of every $s\in S$. An~orientation of a subset of~$S$ is a {\em partial orientation\/} of~$S$.
We write~$\tau(s)$ for the unique orientation of~$s$ contained in~$\tau$,%
   \COMMENT{}
   and say that~$s$ {\em distinguishes\/} two partial orientations~$\tau,\tau'$ of~$S$ if both are defined on~$s$ and $\tau(s)\ne\tau'(s)$.%
   \COMMENT{}

If $\vr\geq\vs$ we say that $\vr$ \emph{points towards} $s$ (and that $\rv$ \emph{points away from} $s$). We say that $\vr$ \emph{points towards} an oriented separation $\vs$ whenever it points towards $s$, i.e., if $\vr \geq \vs$ or $\vr \geq \sv$, and similarly for `points away from'. 
A \emph{star\/} is a set~$\sigma$ of non-degenerate oriented separations that point towards each other. As is easy to check, this happens if and only if $\vr\geq\sv$ (and hence $\vs\geq\rv$) for all  distinct $\vr,\vs\in\sigma$.\looseness=-1

\begin{figure}[ht]
 \center\vskip-6pt
   \includegraphics[scale=1]{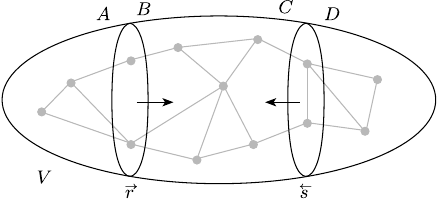}
\vskip-6pt
\caption{\small Nested separations $r = \{A,B\}$ and $s = \{C,D\}$ of a graph. Their orientations $\vr = (A,B)$ and $\sv = (D,C)$ point towards each other, since $\vr\ge \vs$ (as $B\supseteq D$) and $\sv\ge\rv$ (as $C\supseteq A$).}
\label{fig:nested}\vskip-3pt
\end{figure}

Two separations $r,s\in S$ are \emph{nested} if they have orientations that are comparable under~$\leq$; otherwise they~{\em cross\/}. Oriented separations are \emph{nested} if their underlying unoriented separations are nested. A subset of~$\vS$ is \emph{nested} if its elements are pairwise nested.

A subset of~$\vS$ is \emph{consistent} if no pair of its elements $\vr,\vs$ with $r \neq s$ point away from each other. Stars are examples of consistent nested sets of oriented~separations.

We shall often be interested in consistent orientations of~$S$. For each of its elements~$\vr\!$, a~consistent orientation of~$S$ will also contain every $\vs > \vr$ other than, possibly, $\vs=\rv$.%
   \COMMENT{}
   Consistent partial orientations of~$S$ are easily seen to extend to consistent orientations of~$S$, unless they contain a separation that is co-trivial in~$\vS$; see~\cite{ASS}.%
   \COMMENT{}

If $\sigma\subseteq\vS$ is consistent, we say that $\vs\in\vS$ is \emph{required} by $\sigma$ if $\vs \notin \sigma$ and $\sigma\cup\{\sv\}$ is inconsistent. We shall see in \cref{lem:closure} that, pathological cases aside, $\sigma \cup \{\vs \}$ will then be consistent. The \emph{closure} of $\sigma$ is
$$
\lfloor\sigma\rfloor := \sigma \cup\{\vs\in\vS:\text{\(\vs\) is required by \(\sigma\)}\}.
$$
Note that $\sigma$ requires $\vs \notin \sigma$ if and only if there exists an $\vr \in \sigma$ such that $r \neq s$ and $\vr < \vs$. Thus
 $$\lfloor \sigma \rfloor = \sigma \cup \{\, \vs \in \vS: \exists \vr \in \sigma \text{ such that } r \neq s \text{ and }\vr < \vs\}\text{,}$$
 which motivates the term of (upward) `closure'. If $\ucl(\sigma)$ is consistent, then $\lfloor \lfloor \sigma \rfloor \rfloor$ is (defined and) easily shown to be equal to~$\ucl(\sigma)$.%
   \COMMENT{}

If $\sigma$ contains no small separations, the expression above simplifies to
 $$\ucl(\sigma) = \{\,\vs\in\vS : \exists \vr\in\sigma\text{ such that }\vr\le\vs\}.$$
   Indeed, any~$\vs$ in this latter set either lies in~$\sigma$ or there is some $\vr\in\sigma$ such that $\vr < \vs$; in that case $r\ne s$, since otherwise $\sv = \vr < \vs$, making $\vr\in\sigma$  small.

\begin{lemma}[\cite{TSTs}]\label{lem:closure}
Let $\sigma\subseteq\tau\subseteq\vS$ be consistent sets. Then $\lfloor\sigma\rfloor\subseteq\lfloor\tau\rfloor$. If $\tau$ is an orientation of all of $S$, then $\ucl(\tau) = \tau$. If $\tau$ has no elements that are co-trivial in~$\vS$, then $\lfloor\tau\rfloor$ is consistent and $\ucl(\tau)\setminus\tau$ contains at most one orientation of any $s\in S$.%
   \COMMENT{}
\end{lemma}

Consistent orientations of~$S$ can contain small separations. But examples are rare and can be counter-intuitive, so we often exclude them.%
   \COMMENT{}
   Co-trivial separations cannot occur in consistent orientations of~$S$. Indeed if $\sv$ is co-trivial, witnessed by~$r\in S$, then every orientation of~$S$ will have to orient~$r$ too. But it cannot do so consistently with~$\sv$, since both $\vr$ and~$\rv$ are inconsistent with~$\sv$.%
   \COMMENT{}
    Similarly, $\ucl(\{\sv\})$~is then inconsistent, since it contains $\vr$ and~$\rv$, which are both inconsistent with~$\sv$.

An \emph{order function} on $S$ is any map $S\to\mathbb{R}$. Unless otherwise mentioned, we denote such order functions as $s\mapsto |s|$. We extend them to~$\vS$ by letting $|\vs| := |\sv| := |s|$. A~separation system~$(\vS,\le,^*)$ given with an order function on~$S$ is an {\em ordered separation system\/}. For every $k\in\mathbb{R}$ we let
  $$\vSk := \{\vs\in\vS: |s|<k\};$$
this is again an ordered separation system.

We are sometimes interested in orientations of~$S$ that do not have certain subsets. We typically collect those together in some set~$\F$, whose elements we call \emph{forbidden subsets\/}. Formally, if $\F$~is any set,%
   \COMMENT{}
   we say that $\tau\subseteq\vS$ {\em avoids\/}~$\F$ if $\tau$ has no subset in~$\F$, i.e., if no subset of~$\tau$ is an element of ~$\F$.

\begin{definition}\label{def:tangle}
An $\F$\emph{-tangle} of $S$ is an $\F$-avoiding consistent orientation of~$S$. The $\F$-tangles of the subsets~$S_k$%
   \COMMENT{}
   of~$S$ are the $\F$-tangles \emph{in}~$\vS$.
\end{definition}

\noindent
   When our choice of~$\F$ is clear from the context, or arbitrary, we shall also refer to $\F$-tangles simply as {\em tangles\/}.

\medbreak

The tangles of the sets~$S_k\subseteq S$ are the {\em $k$-tangles of~$S$}, or the {\em tangles of order~$k$} in\,$\vS$. When we speak of  {\em maximal\/} tangles in~$\vS$, we refer to their partial ordering as subsets of~$\vS$. A~maximal $\F$-tangle in~$\vS$, thus, is an $\F$-tangle~$\tau$ of some~$S_k$ for which there is no $\F$-tangle $\tau'\ne\tau$ of any~$S_\ell$ such that $\tau = \tau'\cap\vSk$.

\medskip\section{Tangle structure trees}\label{sec:TSTs}

\noindent
   We adopt the graph-theoretic terminology of \cite{DiestelBook25}. Given a tree~$T$ with a root~$r$, we write $\le_r$ for the associated partial ordering on the nodes of~$T$. Maximal elements, including the root if $|T|=1$, are called \emph{leaves}. Any direct successors in~$\le_r$ of a node of~$T$ are its \emph{children}. We write $E_v$ for the set of edges of~$T$ from a node~$v$ to its children.

Let $(\vS,\le,^*)$ be a separation system, and let $\F$ be any set. A \emph{separation~tree} $(T,r,\beta)$ on~$\vS$%
   \COMMENT{}
   consists of a rooted tree $(T,r)$ and an edge labelling $\beta\colon E(T)\to\vS$ such that for every non-leaf ${v\in V(T)}$ there exists a separation $s_v\in S$ such that $\beta$ restricts to a bijection $E_v\to\{\vsv,\svv\}$ and $s_u\neq s_v$ whenever $u <_r v$.\looseness=-1

Thus, every non-leaf node $v$ of a separation tree has either one or two children.  If there is no need to refer to $r$ or $\beta$ explicitly, we usually abbreviate $(T, r, \beta)$ to~$T\!$. For every node~$v$, we write~$\beta_v$ for the set~$\beta\big(E(rTv)\big)$ of edge labels on the path~$rTv$.
A~separation tree is \emph{consistent} if $\beta_v \subseteq \vS$ is consistent for every node $v \in T$. 

A leaf $\ell$ of a consistent separation tree on~$\vS$ is an \emph{$\F$-tangle leaf} if the closure $\lfloor\beta_{\ell}\rfloor$ of $\beta_{\ell}$ is an $\F$-tangle of~$S$. A~non-leaf node~$v$ is an $\F$-{\em tangle node\/} if there is an $\F$-tangle leaf~$\ell\ge_r w$ for every (one or two) successor~$w$ of~$v$ in~$<_r$. Tangle nodes with two children, thus, are precisely the $<_r$-infima of pairs of distinct tangle leaves.\looseness=-1

A~leaf $\ell$ is \emph{forbidden} (by~$\F$) if $\beta_{\ell}$ contains an element of~$\F$ as a subset. Note that tangle leaves are never forbidden. An \emph{$\F$-tangle structure tree of $\vS$} is a consistent separation tree on~$\vS$ in which every leaf is either a tangle leaf or forbidden, and for every non-leaf node~$v$ the set $\beta_v$ has no subset in~$\F$.

\begin{theorem}[{\cite[Theorem~3.7]{TSTs}}]\label{thm:display}
    Let $\vS$ be a separation system, let $\F$ be any set, and let $T$ be an $\F$-tangle structure tree of $\vS$. Then for every $\F$-tangle $\tau$ of $S$ there is a unique leaf $\ell =: \ell(\tau)$ of~$T$ with $\lfloor \beta_{\ell} \rfloor = \tau$. In particular, if all the leaves of $T$ are forbidden, then $S$ has no $\F$-tangle.
\end{theorem}

\noindent
   $\F$-tangle structure trees will be our main tool, so let us see when they exist.

\medbreak

Essentially, two conditions on~$\F$ are needed to ensure that $\vS$ has an $\F$-tangle structure tree. The first is that $\{\sv\}\in\F$ for every $\vs \in \vS$ that is trivial in~$\vS$.  This holds for all relevant~$\F$ (see~\cite{TSTs} for why), and if it does we call ~$\F$ \emph{standard for~$\vS$}.

The second condition is the one that bites. If $T$ is any separation tree on~$\vS$,%
   \COMMENT{}
   then every orientation~$\tau$ of~$S$ contains a set~$\beta_\ell$ for a unique leaf~$\ell$ of~$T$~\cite[Lemma~3.2]{TSTs}. If~$\tau$ is consistent and $\ucl(\beta_\ell)$ orients all of~$S$, then $\ucl(\beta_\ell) = \tau$ by~\cref{lem:closure}.%
   \COMMENT{}
   Hence if $\tau$ is not an $\F$-tangle, then $\ucl(\beta_\ell)$ has a subset~$\sigma$ in~$\F$ that witnesses this.

If $T$ is even an $\F$-tangle structure tree, then $\ell$ is a forbidden leaf, and so we can find such a subset $\sigma\in\F$ of~$\tau$ not just in~$\ucl(\beta_\ell)$ but even in~$\beta_\ell$: among the edge labels of~$T$. In order for this to be possible, we therefore need to make some `richness' assumption about~$\F$: an assumption which, in our example, ensures that $\F$ has enough elements to contain a subset also of~$\beta_\ell$ as soon as it contains a subset of~$\ucl(\beta_\ell)$.\looseness=-1

The main contribution of~\cite{TSTs} was to identify such a richness condition on~$\F$ that is, essentially, both necessary and sufficient for $\F$-tangle structure trees to exist.

To motivate it, let us first state a stronger condition that is less technical, and which clearly implies that $\F$ contains a subset of~$\beta_\ell$ whenever it contains a subset of~$\ucl(\beta_\ell)$. This is that $\F$ is {\em closed under minimization\/} in every consistent orientation~$\tau$ of~$S$: that it contains every set~$\sigma'\subseteq\tau$ obtained from some $\sigma\subseteq\tau$ in~$\F$ by replacing every $\vs\in\sigma$ with some ${\vsdash\le\vs}$ from~$\tau$. Our richness notion in \cref{def:rich} will be weaker than this (\cite[Lem\-ma~4.4]{TSTs}), but still strong enough to ensure the existence of $\F$-tangle structure trees.

A~separation tree $T$ on an ordered separation system~$\vS$ is \emph{ordered} if ${|s_v| \leq |s_w|}$ whenever $v$ and $w$ are non-leaves of $T$ with $v \leq w$. It is \emph{thoroughly ordered} (in~$\vS$) if, for every non-leaf node~$v$, the separation~$s_v$ is not oriented by~$\lfloor \beta_v \rfloor$ (i.e., $\lfloor \beta_v \rfloor$ contains neither~$\vsv$ nor~$\svv$) and has minimum order among the separations not oriented by~$\lfloor \beta_v \rfloor$. Every thoroughly ordered separation tree is ordered~\cite[Lemma~4.1]{TSTs}.

We say that $\vs \in \vS$ is \emph{weakly eclipsed} by $\vr \in \vS$ if $\vr < \vs$ and $|r| \leq |s|$, and \emph{eclipsed} by $\vr$ if $\vr < \vs$ and $|r| < |s|$. Given any set $\tau \subseteq \vS$, a subset $\sigma \subseteq \tau$ is \emph{efficient} (in~$\tau$) if no element of~$\sigma$ is eclipsed by any other element of~$\tau$. It is \emph{strongly efficient} if no element of~$\sigma$ is weakly eclipsed by any other element of~$\tau$. Our separation tree~$T$ is \emph{efficient} if for every leaf~$\ell$ the set~$\beta_{\ell}$ is efficient in~$\lfloor \beta_\ell \rfloor$. Thoroughly ordered separation trees are efficient~\cite[Lemma~4.9]{TSTs}.

\begin{definition}\label{def:rich}
A set $\F$ is \emph{rich for~$\vS$} if every consistent orientation of $S$ that has a subset in~$\F$ also has a strongly efficient\footnote{in this orientation of~$S$} subset in~$\F$.%
   \COMMENT{}
\end{definition}

We have the following two theorems for the existence of $\F$-tangle structure trees.

\begin{theorem}[{\cite[Theorems~4.6,\,4.8]{TSTs}}]\label{thm:generalduality}
Let $\vS$ be an ordered separation system, and let $\F$ be any set that is rich and standard for~$\vS$.  Then there exists a thoroughly ordered $\F\!$-tangle structure tree of~$\vS$. It is unique if the order function on~$S$ is injective.\looseness=-1
\end{theorem}

\noindent
   Note that the uniqueness part of \cref{thm:generalduality} does not require the existence part: an injectively ordered separation system clearly has at most one thoroughly ordered $\F$-tangle structure tree, regardless of what~$\F$ is.%
   \COMMENT{}
   In particular, $\F$~need not be standard or rich. But we know from \cite[Theorem~4.8]{TSTs} that~$\F$, assuming it is standard, will be rich for~$\vS$ if the order function on~$S$ is injective and an $\F$-tangle structure tree exists.\looseness=-1

\medbreak

Thoroughly ordered structure trees can be large. But there is a way to reduce them. Given an $\F$-tangle structure tree~$(T,r,\beta)$, consider an edge~$vw$ of~$T$ such that $w$ is a child of~$v$. We call $vw$ \emph{necessary\/} for a tangle leaf $\ell\ge w$ of~$T$ if $\beta(vw)$ is a $\le$-minimal element of the tangle~$\ucl(\beta_\ell)$, or equivalently, of~$\beta_\ell$. We call~$vw$ \emph{necessary\/} for a forbidden leaf~$\ell\ge w$ if every subset of~$\beta_\ell$ in~$\F$ contains~$\beta(vw)$. A~node~$v$ of~$T$ is \emph{necessary in~$T$} if for every child~$w$ of~$v$ there exists a leaf $\ell \geq w$ such that $vw$ is necessary for~$\ell$. If every node of~$T$ is necessary, then $T$ is \emph{irreducible}.

It is easy to make a tangle structure tree irreducible by contracting edges at unnecessary nodes and deleting certain subtrees~\cite[Section~5]{TSTs}. This preserves efficiency%
   \COMMENT{}
   and being ordered, but not being thoroughly ordered. Our second existence theorem for $\F$-tangle structure trees gains efficiency and irreducibility over that from \cref{thm:generalduality} at the expense of losing the thoroughness of its ordering:

\begin{theorem}[{\cite[Theorem~6.1]{TSTs}}]\label{thm:effandirreducible}
    Let $\vS$ be an ordered separation system, and let $\F$ be any set that is rich and standard for $\vS$.  Then $\vS$ has an efficient and irreducible ordered $\F$-tangle structure tree. 
\end{theorem}

\noindent
   See~\cite[Theorem~6.2]{TSTs} for a long list of properties of these $\F$-tangle structure trees that make them useful for tangle analysis, both in theory and in applications.

\medskip\section{Tangle-tree duality}\label{sec:TTD}

\noindent
   If we apply \cref{thm:effandirreducible} to a separation system~$\vS$ that has no $\F$-tangles, it tells us that there is an $\F$-tangle structure tree all whose leaves are forbidden. Let us call such structure trees {\em $\F$-trees\/}. Conversely, if there is an $\F$-tree of~$\vS$ then $S$ has no $\F$-tangle, by \cref{thm:display}.

\cref{thm:effandirreducible} thus implies%
   \COMMENT{}
   the following dichotomy:

\begin{theorem}[{\cite[Theorem~6.4]{TSTs}}]\label{thm:generalduality2}
Let $\vS$ be an ordered separation system, and let~$\F$ be standard and rich for $\vS$. Then exactly one of the following assertions holds:\looseness=-1
\begin{enumerate}\itemsep2pt\vskip2pt
  \item\label{item:gd1} there exists an $\F$-tangle of $S$;
  \item\label{item:gd2} there exists an $\F$-tree of $\vS$.
\end{enumerate}
In the case of {\rm\labelcref{item:gd2}}, the $\F$-tree can be chosen to be ordered, irreducible and efficient.
\end{theorem}

\noindent
   Note that, unlike in all known theorems of this type, the set~$\F$ of forbidden subsets need not consist of stars of separations (see \cref{sec:basics}).

\medbreak

In the special case that $\F$ does consist of stars of separations, we shall see in \cref{cor:nested} that~-- as long as the $\F$-tree in~\ref{item:gd2} is irreducible, which the theorem says we may require~-- the separations in $\beta(E(T))\subseteq\vS$ will be nested. In contexts where separation systems describe genuine separations, of a set or of a graph, say, nested sets of separations are known to cut up the structure they separate in a tree-like way~-- for example, by a tree-decomposition of the graph. In this way, $\F$-trees impose a tree structure on such a set or graph, which in turn implies that its separations we are considering have no $\F$-tangle.

This phenomenon is one of the corner stones of classical tangle theory, known as {\em tangle-tree duality\/}. Our aim now is to derive such a dichotomy theorem also in our context, as an application of our more comprehensive tangle structure trees. Our version, \cref{cor:E} below, will not imply the classical tangle-tree duality theorem for abstract separation systems~\cite{TangleTreeAbstract}, because our richness assumptions made for~$\F$ are different. But it will offer the same dichotomy. And we shall see in~\cite{TSTapps} that it implies all known instances of the classical theorem for concrete structures, such as graphs or matroids.

\medbreak

To carry this out formally, let us introduce the terms in which classical tangle-tree duality is cast. Let $\vE(T)$ denote the set of orientations of the edges of a tree~$T$. For an oriented edge~$\ve= (t,t')$ of~$T$ we call~$t =:i(\ve)$ its \emph{initial node} and $t' =: t(\ve)$ its \emph{terminal node}, and we denote its inverse~$(t',t)$ by~$\ev$. We think of $t(\ve)$ as the node that $\ve$ \emph{points towards}. There is a natural partial ordering~$\le$ on $\vE(T)$ in which $\ve > \vf$ if $f\ne e$ and the unique path in~$T$ from $e$ to~$f$ starts at~$t(\ve)$ and ends at~$i(\vf)$.%
   \COMMENT{}
   With the involution~$^*\colon\ve\mapsto\ev$, this makes $(\vE(T),\le,^*)$ into a separation system.

Since every two edges $e,f$ of~$T$ have comparable orientations, the elements of~$\vE(T)$ are all nested. And for every node~$t$ of~$T$, the set
\[
\vF\!_t:=\{\,\ve\in\vE(T): t(\ve)=t\,\}
\]
is a star in this separation system.

Now let $(\vS,\le,^*)$ be any other separation system~-- typically, but not necessarily, consisting of separations of a set or graph. An \emph{$S$-tree} $(T,\alpha)$ is a tree $T$ together with a map $\alpha\colon \vE(T)\to\vS$ that satisfies $\alpha(\ev)=\alpha(\ve)^*$ for every $\ve\in\vE(T)$. It is an $S$-tree \emph{over\/} a set~$\F$ if $\alpha(\vF\!_t)\in\F$ for every $t\in V(T)$.%
   \COMMENT{}

If $\alpha$ preserves the partial ordering~$\le$ on~$\vE(T)$, i.e.\ if $\ve\le\vf$ implies $\alpha(\ve)\le\alpha(\vf)$,%
   \COMMENT{}
   and the image of~$\alpha$ contains no degenerate separations, then the sets $\alpha(\vF\!_t)\in\F$ will be stars in~$\vS$, because the~$\vF\!_t$ are stars in~$\vE(T)$.%
   \COMMENT{}
   Conversely:

\begin{lemma}[\cite{TreeSets}]\label{lem:starsimpliesnested}%
   \COMMENT{}
Let $(T,\alpha)$ be an $S$-tree over a set stars.%
   \COMMENT{}
   If $\alpha$ is injective on~$\vF\!_t$ for every node~$t$ of~$T$, then $\ve\le\vf$ implies $\alpha(\ve)\le\alpha(\vf)$ for all $\ve,\vf\in\vE(T)$. In particular, $\alpha(\vE(T))$ is a nested subset of~$\vS$.

Conversely, for every nested, regular, finite separation system~$(\vS,\le,^*)$ there is an $S$-tree $(T,\alpha)$ over stars such that $\alpha(\vE(T)) = \vS$ and $\alpha$ is injective on the stars~$\vF\!_t$.
\end{lemma}

\noindent
We remark that if $\alpha$ is not injective on some~$\vF\!_t$, then $\alpha(\vE(T))$ can fail to be nested.%
   \COMMENT{}

\medbreak

\cref{lem:starsimpliesnested} motivates the study of $S$-trees over sets~$\F$ of stars (of separations): they make precise, in a more general way than \td s do, the `tree-like way' in which nested sets of separations cut up any under\-lying structure.

The existence of an $S$-tree $(T,\alpha)$ over a set~$\F$, not necessarily consisting of stars, precludes the existence of an $\F$-tangle of~$S$. 
Indeed, given any orientation $\tau$ of~$S$, the set $\alpha^{-1}(\tau)$%
   \COMMENT{}
   is an orientation of the edges of~$T$. Since trees contain no cycles, this orientation will contain~$\vF\!_t$ for some~$t$, so that $\alpha(\vF\!_t)\subseteq\tau$. But as $(T,\alpha)$ is an $S$-tree over~$\F$ we have $\alpha(\vF\!_t)\in\F$, so $\tau$ is not an $\F$-tangle.

The converse implication, that there is an $S$-tree over~$\F$ whenever $S$ has no $\F$-tangle, holds only when $\F$ consists of stars. But assuming this is not enough. In~\cite{TangleTreeAbstract}, which offers the most general theorem of this type to date, there is the further assumption that $\vS$ must be \emph{$\F$-separable}. This is a richness condition on~$\F$, different from ours, which requires~$\F$ to also contain stars which its elements `induce on the sides' of certain separations in~$S$. The precise notion of $\F$-separability is quite technical, but immaterial here, so we refer the reader to~\cite{TangleTreeAbstract,TangleTreeGraphsMatroids} for details.

The following dichotomy between $\F$-tangles and $S$-trees over~$\F$ is the classical tangle-tree duality theorem for abstract separation systems:

\begin{theorem}{\cite{TangleTreeAbstract}}\label{thm:prevduality}
Let $\vS$ be a separation system inside a universe~$\vU$ of separations,%
   \COMMENT{}
   and let $\F\subseteq 2^{\vS}\!$%
   \COMMENT{}
   be a set of stars that is standard for $\vS$. 
If $\vS$ is $\F$-separable, then exactly one of the following assertions holds:
\begin{enumerate}\itemsep2pt\vskip2pt
  \item there exists an $\F$-tangle of $S$;
  \item there exists an $S$-tree over $\F$.
\end{enumerate}
\end{theorem}

When $\F$ consists of stars, we can compare \cref{thm:prevduality} with what our more general \cref{thm:generalduality2} says for such~$\F$. Both theorems guarantee that if $S$ admits no $\F$-tangle then it has an easily detectable certificate in the form of a certain tree-like structure: an $S$-tree over $\F$ in the case of \cref{thm:prevduality}, and an $\F$-tree of $\vS$ in the case of our new \cref{thm:generalduality2}.

The following result shows that our \cref{thm:generalduality2} is not only more general, because its~$\F$ need not consist of stars, but also stronger than \cref{thm:prevduality}, because we can obtain $S$-trees over~$\F$ from $\F$-trees when $\F$ does consist of stars. In fact, {\em every\/} irreducible $\F$-tree of~$\vS$ then has the property that its edge labels in~$\vS$ are precisely the edge labels of some $S$-tree over~$\F$:

\begin{theorem}\label{thm:main}
Let $\vS$ be a separation system without trivial elements, ${\F\subseteq 2^{\vS}}\!\!$ a~set of stars, and $(T,r,\beta)$ an irreducible $\F$-tree of~$\vS$. Then there exists an $S$-tree $(T', \alpha)$ over $\F$ and a map $\gamma\colon V(T') \cup \vE(T') \to V(T) \cup E(T)$ such that%
   \COMMENT{}
\begin{enumerate}\itemsep2pt\vskip4pt
    \item\label{item:main1} $\gamma$ is a bijection from the nodes of $T'$ to the leaves of $T$;
    \item\label{item:main2} $\gamma$ is a bijection from the oriented edges of $T'$ to the edges of~$T$;
    \item\label{item:main3} for every edge~$e$ of~$T'\!$ there is a node~$v_e$ of~$T\!$ such that $\{\gamma(\ve), \gamma(\ev)\} = E_{v_e}$, and these~$v_e$ are distinct for different~$e$;
    \item\label{item:main4} for every oriented edge $\ve$ of $T'$, the map $\gamma$ sends $t(\ve)$ to a leaf $\ell > v_e$ of $T$ such that $\gamma(\ve)$ is the first edge of the path $v_eT\ell$;
    \item\label{item:main5} $\alpha = \beta \circ \gamma$.
\end{enumerate}\vskip2pt
In particular, $\beta(E(T)) = \alpha(\vE(T'))$. 
\end{theorem}

Before we prove \cref{thm:main}, let us note a corollary for $\F$-trees that seems remarkable in its own right, quite independently of tangle-tree duality. If~$\F$ consists of stars, the edges of any irreducible $\F$-tree map to nested separations in~$\vS$:

\begin{corollary}\label{cor:nested}
Let $\vS$ be a separation system without trivial separations, and let $\F\subseteq 2^{\vS}\!$ be a set of stars. Then for every irreducible $\F$-tree $(T,r,\beta)$ of~$\vS$ the image $\beta(E(T))$ of its edges is a nested set of separations in~$\vS$.
\end{corollary}

\begin{proof}
$\!$Our aim is to apply \cref{lem:starsimpliesnested} to the $S$-tree $(T'\!,\alpha)$ provided by \cref{thm:main},\penalty-200 to show that $\alpha(\vE(T')) = \beta(E(T))$ is nested. To apply the lemma, we have to check that~$\alpha$ is injective on the stars~$\vF\!_t$ in~$\vE(T')$.

Given a node~$t$ of~$T'\!$, we have $t = t(\ve)$ for all $\ve\in\vF\!_t$. For $\ell:=\gamma(t)$ we therefore have $\gamma(\vF\!_t)\subseteq E(rT\ell)$ by \cref{thm:main}\,\ref{item:main4}.%
   \COMMENT{}
   In order to show that $\alpha$ is injective on~$\vF\!_t$, it thus suffices by \ref{item:main5} $\alpha=\beta\circ\gamma$ and the injectivity of~$\gamma$ to show that~$\beta$ is injective on~$E(rT\ell)$. This, however, follows from the last clause in the definition of separation trees.%
   \COMMENT{}
\end{proof}

\begin{proof}[Proof of \cref{thm:main}]
We start our construction of~$T'$ by taking as~$V(T')$ any set we can map to the leaves of~$T$ by a bijection~$\gamma$. This ensures~\ref{item:main1}.

We shall give~$T'$ one edge~$e$ for every non-leaf node $v =: v_e$ of~$T$. To choose the ends of~$e$ in~$T'$, consider the children $w_1$ and~$w_2$ of~$v$ in~$T$.%
   \footnote{These are distinct: $s_v$~cannot be degenerate, because $vw_1$ is necessary for a forbidden%
   \COMMENT{}
   leaf $\ell\ge w_1$ of~$T$, so $\beta(vw_1)$ lies in a star $\sigma\in\F$. Stars do not contain degenerate separations.}%
   \COMMENT{}
   For each $i\in\{1,2\}$, since $T$ is irreducible, the edge~$vw_i$ is necessary for some leaf~$\ell_i\ge w_i$ of~$T$.%
   \footnote{Our proof will imply that these~$\ell_i$ are unique, but for the definition it suffices to pick any.}%
   \COMMENT{}
   Let $e$ join the nodes $\gamma^{-1}(\ell_1)$ and~$\gamma^{-1}(\ell_2)$ of~$T'$, and let $\gamma(\ve) := vw_i$ for the orientation~$\ve$ of~$e$ for which $\gamma(t(\ve)) = \ell_i$. This satisfies \ref{item:main2}--\ref{item:main4}, if we think of~$T'$ as a multigraph for now, i.e., allow multiple edges.

However, since $v_e$ is the unique infimum of~$\ell_1$ and~$\ell_2$ in~$T$, the pair $\gamma^{-1}(\{\ell_1,\ell_2\})$ of nodes in~$T'$ is joined only by the edge~$e$: the graph~$T'$ has no parallel edges and no loops.
Finally, now that $\gamma$ is fixed, let $\alpha:=\beta\circ\gamma$ as in~\ref{item:main5}.

It remains to show that $T'$ is a tree, and that $(T',\alpha)$ is over~$\F$: as $\alpha(\ve)^* = \alpha(\ev)$ by definition of $\gamma$ and~$\alpha$, it will be clear that $(T',\alpha)$ is an $S$-tree once we know that $T'$ is a tree.

Let us begin by showing that $T'$ is acyclic. Consider two edges $e, f$ of $T'$ with $t(\ve) = i(\vf)$. By~\labelcref{item:main4}, the edges $\gamma(\ve)$ and $\gamma(\fv)$ of~$T$ at~$v_e$ and~$v_f$, respectively, both lie on the path~$rT\ell$ for $\ell=\gamma(t(\ve))=\gamma(i(\vf))$. As $e \neq f$, the nodes~$v_e$ and~$v_f$ of~$rT\ell$ are distinct,%
   \COMMENT{}
   by~(iii), so $v_e < v_f$ or $v_f < v_e$. As $s_u \neq s_v$ for $u < v$ in any separation tree, $\alpha(\ve)=\beta(\gamma(\ve))$ and $\alpha(\fv) = \beta(\gamma(\fv))$, compare~\labelcref{item:main5}, are thus orientations of distinct separations $s_{v_e}\!\ne s_{v_f}\!$ in~$S$.

Since we chose~$\gamma^{-1}(\ell)$ as~$t(\ve)$ and as~$t(\fv)$ when we picked the ends of the edges~$e$ and~$f$ in~$T'$, the edges $\gamma(\ve)$ and $\gamma(\fv)$ of~$T$ are necessary for~$\ell$: their $\beta$-values $\alpha(\ve)$ and $\alpha(\fv)$ each lie in some $\sigma \in \F$ that is a subset of~$\beta_{\ell}$, and every subset of~$\beta_{\ell}$ in~$\F$ contains both. Consider any such $\sigma\in\F$. Since $\sigma$ is a star, we have $\alpha(\ve) \geq \alpha(\fv)^* = \alpha(\vf)$. As $s_{v_e}\!\ne s_{v_f}$, the above inequality is strict. The $\alpha$-values along any oriented path in~$T'$ thus strictly decrease.%
   \COMMENT{}
   Hence $T'$ contains no (oriented) cycles.

To complete our proof that $T'$ is a tree, it suffices to show that it has one more node than edges. The first of these numbers equals the number of leaves of $T$; the second equals the number of non-leaves of $T$. This latter number is one less than the former, since $T$ is a binary tree. 

To show that $(T', \alpha)$ is an $S$-tree over~$\F$, consider any node~$t$ of~$T'$; we have to show that $\alpha(\vF\!_t) \in \F$. As earlier,%
   \COMMENT{}
   all the edges $\ve\in\vF\!_t$ are such that $\gamma(\ve)\in rT\ell$ for $\ell := \gamma(t)$ is necessary for~$\ell$, so $\alpha(\vF\!_t) = \beta(\gamma(\vF\!_t)) \subseteq\sigma\subseteq\beta_\ell$ for some $\sigma\in\F$. It remains to show that $\alpha(\vF\!_t)\supseteq \sigma$.

Suppose there exists an $\vs \in \sigma\setminus \alpha(\vF\!_t)$. As $\vs\in\sigma\subseteq\beta_\ell = \beta(rT\ell)$, we can find an edge $vw$ in~$rT\ell$ with $v<w$ such that $\beta(vw) = \vs$. Let $\vf := \gamma^{-1}(vw)$. Then $\alpha(\vf)=\vs$, and $\gamma(\vf)\in E(rT\ell)$ but $\gamma(\vf)\notin\gamma(\vF\!_t)\subseteq E(rT\ell)$. In particular, $v_f\ne v_e$ for every $\ve\in\vF\!_t$, and hence $f\ne e$, so neither $\vf$ or~$\fv$ lies in~$\vF\!_t$. But note that $v_f$ and all these~$v_e$ lie on the path~$rT\ell$, so they are comparable in the tree-order on~$V(T)$.

Since $T'$ is a tree, either $\vf$ or $\fv$ points towards~$t$ in~$T'$. Suppose first that $\vf$ does. Then $\vf\ge \ve$%
   \COMMENT{}
   in~$\vE(T')$ for some $\ve \in \vF\!_t$. As in our acyclicity proof earlier, this implies that $\alpha(\vf) \geq \alpha(\ve)$.%
   \COMMENT{}
   As both $\vs = \alpha(\vf)$ and~$\alpha(\ve)$ lie in the star~$\sigma$, we have $\alpha(\vf)\ge\alpha(\ve)^*$ too.%
   \COMMENT{}
   As $v_f$ and~$v_e$ are distinct nodes on~$rT\ell$, which implies that $s=s_{v_f}\ne s_{v_e} = \{\alpha(\ve),\alpha(\ve)^*\}$ since $T$ is a separation tree, $\vs = \alpha(\vf)$ is thus trivial in~$\vS$ witnessed by~$s_{v_e}$, a~contradiction to our assumptions about~$\vS$.

Suppose now that the edge~$\fv$ points towards~$t$ in~$T'$. Then $\fv \geq \ve\in\vF\!_t$ and hence $\alpha(\vf)^*\ge\alpha(\ve)$, while $s=s_{v_f}\ne s_{v_e} = \{\alpha(\ve),\alpha(\ve)^*\}$, as earlier. This makes $\vs = \alpha(\vf)$ and $\alpha(\ve)$ inconsistent elements of $\beta_{\ell}$, a~contradiction to the fact that $T$ is a consistent separation tree.
\end{proof}

Let us discuss briefly what we can say if $\vS$ does have trivial elements. The proof given above stands until we have to prove that the $S$-tree~$T'\!$ we have constructed is over~$\F$. If $\vr\in\vS$ is trivial in~$\vS$ witnessed by~$s\in S$, this can indeed fail. Indeed, assume that $\vs = \alpha(\ve)$ with $\gamma(t(\ve)) = \ell$ and $\gamma(t(\ev)) = \ell'$. Let $v_f$ be the predecessor%
   \COMMENT{}
   of~$v_e$ in~$T$, and assume that $\beta(v_f v_e) = \vr$.
   \COMMENT{}
   Then $v_f v_e$ is the first edge of the path in~$T$ from~$v_f$ to both $\ell$ and~$\ell'$, so either of these would be eligible as~$\gamma(t(\vf))$, where $\vf$ is the orientation of~$f$ with $\alpha(\vf) = \vr$, as long as $v_f v_e$ is necessary for~$\ell$ and~$\ell'$, respectively. This will be the case if and only if $\vr$ lies in the stars $\sigma\in\F$ or~$\sigma'\in\F$ that contain~$\vs$ and~$\sv$, respectively; let us assume it lies in both. (Recall that $\beta^{-1}(\vs)$ and~$\beta^{-1}(\sv)$ are necessary for $\ell$ and~$\ell'$, respectively, so $\vs\in\sigma$ and $\sv\in\sigma'$ for some and for all stars $\sigma\subseteq\beta_\ell$ and $\sigma'\subseteq\beta_{\ell'}$ in~$\F$.) However, when we constructed~$T'$, we chose only one of~$\ell$ and~$\ell'$ as~$\gamma(t(\vf))$, say~$\ell$. We then have $\vr\in\sigma'\setminus\alpha(\vF_{t'})$ for $t' := \gamma^{-1}(\ell')$.\looseness=-1

However, we can easily mend this: just add a new leaf~$t$ to~$T'$ adjacent to~$t'$, with $\alpha(tt') := \vr$. We shall then need that $\{\rv\}\in\F$, but this will be the case if we assume that $\F$ is standard, as we have to in order to obtain~$T$ in the first place. More generally, we can add to~$T'$ such new leaves at all nodes~$t'$ whose~$\alpha(\vF_{t'})$ fails to contain a trivial separation~$\vr$ whose triviality is witnessed by some~$s$ with ${\vs\in\alpha(\vF_{t'})}$. The modified tree will no longer satisfy the detailed clauses in \cref{thm:main}, but it will be an $S$-tree over~$\F$ all whose non-trivial labels $\alpha(\vs)$ are obtained as in our proof.

Note that the `tree-likeness' in $S$-trees over $\F$ and in $\F$-trees, respectively, refers to very different structures: while $S$-trees over~$\F$ in \cref{thm:prevduality} exhibit tree-like\-ness in the (object divided by the) separation system $\vS$ itself -- for example, they define tree-decompositions with parts specified by $\F$ if $S$ consists of separations of a graph -- the $\F$-trees in \cref{thm:generalduality2} are `tree-like' only in the formal sense of a decision-tree for how to orient the elements of $S$. It is all the more remarkable, therefore, that the structuring trees found in~$S$ by \cref{thm:prevduality} can be recovered from the $\F$-trees in \cref{thm:generalduality2}, as shown in \cref{thm:main}. We thus have a new tangle-tree duality theorem: one that offers the same dichotomy as \cref{thm:prevduality}, but with a different premise. 

To state this in optimal form, we need one more definition.
Given any set~$\F \subseteq 2^{\vS}\!$ and any order function on~$S$, let $\F_{\rm eff }$ denote the set of all sets $\sigma \in \F$ that are efficient in their closure $\lfloor \sigma \rfloor$ in~$\vS$%
   \COMMENT{}
   (see \cref{sec:TSTs}). Note that this definition makes sense independently of any tangle structure trees.

\begin{corollary}\label{cor:E}
Let $\vS$ be an ordered separation system without trivial separations. Let $\F \subseteq 2^{\vS}\!$ be a%
   \COMMENT{}
   rich set of stars. Then exactly one of the following assertions holds:\looseness=-1
\begin{enumerate}\itemsep2pt\vskip2pt
  \item\label{item:E1} there exists an $\F$-tangle of $S$;
  \item\label{item:E2} there exists an $S$-tree over $\F$.
\end{enumerate}
In the case of {\rm\labelcref{item:E2}}, the $S$-tree can be chosen to be over~$\F_{\rm eff}$. 
\end{corollary}

\begin{proof}
    As indicated after \cref{lem:starsimpliesnested}, the assertions \labelcref{item:E1} and \labelcref{item:E2} cannot both hold. Let us assume that \labelcref{item:E1} fails, and show \labelcref{item:E2}. By \cref{thm:generalduality2}, there is an $\F$-tree~$T$ of~$\vS$, which may be chosen irreducible, efficient and ordered. If $T$ is efficient, it will be an $\F_{\rm eff}$-tree, since any $\sigma \subseteq \beta_\ell$ from~$\F$ is then efficient in $\lfloor \beta_\ell \rfloor$ and therefore also efficient in $\lfloor \sigma \rfloor\subseteq\ucl(\beta_\ell)$. Now \labelcref{item:E2} and the final statement of the corollary follow from \cref{thm:main}.
\end{proof}

While \cref{thm:prevduality} and \cref{cor:E} proclaim the same tangle-tree dichotomy, the requirements they make on~$\F$ in their premise are different: $\F$-separability of~$\vS$ and richness of~$\F$, respectively. In order to derive \cref{thm:prevduality} formally from \cref{cor:E} (in the case that $\F$ consists of stars) we would need to show that if $\vS$ is $\F$-separable then $\F$ is rich for~$\vS$. The condition of $\F$-separability, however, was tailor-made for the proof of \cref{thm:prevduality}. It therefore seems unlikely that we can formally derive from it that $\F$ is rich. However, in \cref{sec:measure} we shall derive this from a condition that is only slightly stronger than $\F$-separability. This stronger condition, in fact, holds and is used in all the known applications of \cref{thm:prevduality}, rather than $\F$-separability itself. This will be explored in~\cite{TSTapps}.

\medskip\section{Submodular order functions}\label{sec:orderfunctions}

\noindent
   Submodular order functions are a key ingredient to much of tangle theory. We shall need them in this paper too, both for our tree-of-tangles theorems in \cref{sec:ToTs} and to retrieve the essence of the tangle-tree duality theorem from~\cite{TangleTreeAbstract} in \cref{sec:measure}.

In both these applications we shall need tangle structure trees that are not just ordered but thoroughly ordered. These exist reliably only when our order function is injective. If it is not~-- which often happens in theoretical applications%
   \footnote{Order functions used for real-world tangle applications, such as in clustering, tend to be injective. But the standard order function on graph separations, for example, is not.}
   ~-- we have to tweak it slightly to make it so. We shall prove in this section that we can do this while preserving its submodularity.

We say that an order function~$o'$ on a set~$S$ \emph{refines\/} another order function~$o$ on~$S$ if\looseness=-1
$$
o(r) < o(s)\ \Rightarrow\  o'(r) < o'(s)\quad(\forall r,s\in S).
$$
Thus, $o'$ refines~$o$ if and only if $o'(r)\le o'(s)$ implies $o(r)\le o(s)$. But we do not have the same implication for~`$<$', since $o'$ need not be constant on the subsets of~$S$ whose elements have a common value under~$o$.

Let $\vS$ be a separation system with an order function~$o$, and let $\F$ be any set. Recall that the $\F$-tangles of the subsets
 $$S_k = \{\,s\in S\mid o(s) < k\,\}$$
 of~$S$ are the $\F$-tangles {\em in\/}~$\vS$ with respect to~$o$.

\begin{lemma}\label{lem:ordertanglestoorderingtangles}
If $o'$ refines~$o$, then the tangles in~$\vS$ with respect to~$o$ are also tangles in~$\vS$ with respect to~$o'\!$.
\end{lemma}

\begin{proof}
If $o'$ refines~$o$ then, for every $k\in\R$ and with $S_k$ defined with respect to~$o$ as above, the range~$o'(S_k)$ in~$\R$ is an initial segment of~$o'(S)$. For every~$k$ choose $k' \in \mathbb{R}$ larger than every element of $o'(S_k)$ but smaller than every element of $o'(S) \setminus o'(S_k)$. Then $S_k$ can be rewritten as~$\{\,s\in S\mid o'(s)<k'\}$, which puts the $\F$-tangles of~$S_k$ in~$\operatorname{tangle}(\vS, \F, o')$. 
\end{proof}

Let us consider a separation system~$\vU$ that is a universe: one whose partial ordering~$\le$ makes it into a lattice. Any two $\vr,\vs\in\vU$  thus have a supremum~$\vr\lor\vs$ and an infimum~$\vr\land\vs$ in~$\vU$. Let us call a function $u\colon \vU \to \mathbb{R}$ \emph{submodular} if
$$
u(\vr\lor\vs)+u(\vr\land\vs)\leq u(\vr)+u(\vs)
$$
for all $\vr,\vs\in\vU$, and \emph{structurally submodular} if
 $${u(\vr\lor\vs)\le u(\vr)}\quad\text{or}\quad u(\vr\land\vs)\le u(\vs)\eqno{(\dag)}$$
for all $\vr,\vs\in\vU$. Note that submodular functions on~$\vU$ are also structurally submodular, but not conversely. Unlike submodularity, the structural submodularity of a real function on~$\vU$ is preserved when we compose it with an order isomorphism on $\mathbb{R}$, but it is not preserved by summing functions.%
   \COMMENT{}

Let us call an order function~$|\ |$ on the set~$U$ of unoriented separations \emph{submodular} if the function $\vs\mapsto |s|$ it induces on~$\vU$ is submodular, and similarly for `structurally submodular'. As earlier, submodularity of order functions on~$U$ is preserved under summing (but not under composition with order isomorphisms), while structural submodularity of order functions on~$U$ is preserved under composition with order isomorphisms (but not under summing).

Note also that, given any $k\in\R$, the separation system $\vUk = \{\vs\in\vU : |s| < k\}$ is submodular, as defined at the start of \cref{sec:basics}, as soon as $|\ |$ is structurally submodular on~$\vU\!$. Conversely, if $\vS\subseteq\vU$ is submodular, one can define a structurally submodular order function on~$U$ so that $\vS = \vUk$. But there need not exist a submodular order function on~$U$ with this property; see~\cite[Example~5.11]{SARefiningInessPartsExtended}.%
   \COMMENT{}

It is not hard to refine a submodular order function to make it injective. Ideally, we would like to do this while preserving its submodularity. We shall prove a little less than this, which is still good enough for our use of submodularity: that we can refine it to an order function that is still structurally submodular.

\begin{lemma}\label{lem:inversesub}
Let $\vU$ be a separation universe with a submodular function~$u$. Then $u'\colon \vs\mapsto u(\sv)$ and $w\colon s\mapsto u(\vs) + u(\sv)$ are also submodular order functions on~$U\!$.\looseness=-1
\end{lemma}

\begin{proof} Since $u$ is submodular, we also have
    $$u'(\vs \lor \vr) + u'(\vs \land \vr) = u(\sv \land \rv) + u(\sv \lor \rv) \leq u(\sv) + u(\rv) = u'(\vs) + u'(\vr) \text{,}$$%
   \COMMENT{}
so $u'$ is submodular. The sum of two submodular functions is clearly  submodular, hence so is also $w = u + u'$. 
\end{proof}

\begin{lemma}\label{sub:1} Let $\vU$ be a separation system and $\vt\! \in \vU\!$. The function ${\mathbf{1}\!_{\vt}\colon \vU \to \{0, 1\}}$ given by
$$
\mathbf{1}\!_{\vt}(\vs) = 
\begin{cases}
    0\,\text{ if } \vs \leq \vt,\\
    1\,\text{ otherwise.}
\end{cases}
$$ is submodular.
\end{lemma}

\begin{proof}
Let $\vr, \vs \in \vU$. We have to show that $$\mathbf{1}\!_{\vt}(\vr \land \vs) + \mathbf{1}\!_{\vt}(\vr \lor \vs) \leq \mathbf{1}\!_{\vt}(\vr) + \mathbf{1}\!_{\vt}(\vs) \text{.}$$
If the right-hand side of the inequality is~$2$, the inequality holds trivially. So let us assume that $\mathbf{1}\!_{\vt}(\vr) = 0$. Then $\vr \land \vs \leq \vr \leq \vt$, so $\mathbf{1}\!_{\vt}(\vr \land \vs) = 0$. Hence the left-hand side of the inequality is at most~$1$. The inequality thus holds unless the right-hand side is zero, i.e., unless also $\mathbf{1}\!_{\vt}(\vs) = 0$. But then also $\vs \leq \vt$ and thus $\vr \lor \vs \leq \vt$, so $\mathbf{1}\!_{\vt}(\vr \lor \vs) = 0$, making the left-hand side zero too.
\end{proof}

It may help to think of the sums in our next lemma as $n$-ary integer expansions.%
   \COMMENT{} 

\begin{lemma}\label{sub:2} Let $\vU$ be a separation universe, $\iota\colon \vU\to\{0,\dots,|\vU|-1\}$ a bijection, and $n \in \N$. Then
$$
\gamma_n^{\iota}(\vs):=\!\!\!\!\!\sum_{\vt\!\in \vU, \vt\ngeq\vs}\!\!\!\!\!\!\!\! n^{\iota(\!\vt\!)}\> = \sum_{\vt \in \vU} n^{\iota(\!\vt\!)} \mathbf{1}\!_{\vt}(\vs).
$$
In particular $\gamma_n^{\iota}$ is submodular. 
\end{lemma}

\begin{proof} The first part is immediate, as the additional summands in the second sum are all zero.%
   \COMMENT{}
The submodularity of~$\gamma_n^{\iota}$ then follows from \cref{sub:1} and the fact that sums of submodular functions are again submodular.%
   \COMMENT{}
\end{proof}

\begin{lemma}\label{lem:ordertoordering}
Let $\vU$ be a separation universe with a submodular order function~$o$. Then there exists an injective submodular order function~$o'\!$ on~$U\!$ that refines~$o$.
\end{lemma}
\begin{proof}
We shall construct a small injective and submodular perturbation $\delta\colon U\to\mathbb{R}$, so that $o':=o+\delta$ is injective and refines~$o$.

First choose $\varepsilon>0$ as follows. Let $\Delta := \{|o(s)-o(r)| : s,r\in U,\ o(s)\neq o(r)\}$. If $\Delta\neq\emptyset$ set $\varepsilon:=\min\Delta>0$; otherwise set $\varepsilon:=1$. In either case $\varepsilon>0$.

Fix an arbitrary bijection $\iota\colon \vU\to\{0,\dots,|\vU|-1\}$ and set $\gamma(s):=\gamma_3^{\iota}(\vs)+\gamma_{3}^{\iota}(\sv)$. By \cref{lem:inversesub} and \cref{sub:2} the function $\gamma\colon U \to \mathbb{R}$ is submodular. We claim $\gamma$ is injective. Write $\gamma(s)$ in base~$3$:
$$
\gamma(s)=\sum_{k=0}^{|\vU|-1} a_k(s)\,3^k,\qquad a_k(s)\in\{0,1,2\}.
$$
By construction we have $a_k(s) \leq 1$ if and only if the separation $\vt = \iota^{-1}(k)$ points towards~$s$. If $\gamma(s)=\gamma(r)$ then their base-$3$ expansions coincide, hence every $\vt\in \vU$ points towards~$s$ if and only if it points towards~$r$. From this we will now deduce $s=r$, showing that $\gamma$ is injective.

Let $\vso$ be any orientation of~$s$. Suppose that, for some $n \in \N$, we have defined an orientation~$\vsn$ of~$s$. Then since $\vt := \vsn$ points towards~$s$, it also points towards~$r$. So there exists an orientation~$\vrn$ of~$r$ with $\vsn \geq \vrn$. Analogously, given any orientation~$\vrn$ of~$r$ we can find an orientation~$\vsnplusone$ of~$s$ with $\vrn \geq \vsnplusone$. For any orientation~$\vso$ of~$s$ we can thus construct an infinite chain $\vso \geq \vro \geq \vsone \geq \dots$ of, alternately, orientations of $s$ and $r$. As $r$ and $s$ have at most two orientations each, at most two of these inequalities can be strict. Hence for $m$ large enough we have $\vsm = \vrm$, and therefore $s = r$ as desired. This completes our proof that $\gamma$ is injective.

Note that $0\le \gamma(s)<3^{|\vU|}$ for every $s$. Define
$$
\delta(s):=\big({\textstyle \frac12}\varepsilon/3^{|\vU|}\big)\,\gamma(s).
$$
Then $0 \leq \delta(s)<\varepsilon/2$ for all~$s$. Note that $\delta$ is injective and submodular, because $\gamma$~is and ${\textstyle \frac12}\varepsilon/3^{|\vU|} > 0$. 

Set $o':=o+\delta$. As a sum of submodular functions, $o'$~is submodular. To see that $o'$~is injective, suppose $o'(r)=o'(s)$. Then
$$
|o(r)-o(s)| = |\delta(r)-\delta(s)|\leq |\delta(r)|+|\delta(s)|<\varepsilon.
$$
By the choice of $\varepsilon$ we must have $o(s)=o(r)$, and hence $\delta(s)=\delta(r)$. Since $\delta$ is injective this gives $s=r$. Finally, $o'$~refines~$o$, because whenever $o(r) > o(s)$ we have\looseness=-1
$$
o'(r)-o'(s)=o(r)-o(s)+\delta(r)-\delta(s)\geq \varepsilon-|\delta(r)|-|\delta(s)|>0\,,
$$
so $o'(r)>o'(s)$.
\end{proof}

An {\em enumeration\/} of a set~$S$ is any bijection between $S$ and~$\{1,\dots,|S|\}$.

\begin{lemma}\label{lem:newordertoordering}
    Let $\vU\!$ be a separation universe with a submodular order function~$o$. Then there exists a structurally submodular enumeration of~$U\!$ that refines~$o$. 
\end{lemma}%
   \COMMENT{}

\begin{proof} Compose the injective submodular refinement~$o'$ of~$o$ provided by \cref{lem:ordertoordering} with an order isomorphism $\varphi$ from its range to $\{1, \dots, |U|\}$. Since~$o'$, being submodular, is also structurally submodular, and structural submodularity is preserved under composition with order isomorphisms, the enumeration~$\varphi\circ o'$ of~$U$ is structurally submodular, and like~$o'$ it refines~$o$.
\end{proof}

\medskip\section{Trees of tangles}\label{sec:ToTs}

\noindent
   In \cref{sec:TTD} we derived one of the two fundamental types of theorem about tangles from our tangle structure trees: a~so-called tangle-tree duality theorem. In this section we shall do the same for the second pillar of tangle theory: its so-called tree-of-tangles theorems.

Let $\vS$ be an ordered separation system. Recall from \cref{sec:basics} that a separation {\em distinguishes\/} two tangles in~$\vS$ if they orient it differently. Let us say that it distinguishes them {\em optimally\/}%
   \footnote{\lineskiplimit = -3pt The usual terminology is to say `efficiently' rather than `optimally'. We use the latter only in this paper, to avoid confusion with our use of the term `efficient' for subsets of~$\vS$.}
   if it distinguishes them but no separation in~$S$ of lower order does. A~nested set~$N\subseteq S$ is a {\em tree\,%
   \footnote{Recall that nested sets of separations of structures such as sets or graphs divide them in a `tree-like' way~\cite{TreeSets}; hence the `tree' in the name.}\,%
   of the tangles of~$S$} if every two $\F$-tangles of~$S$ (for some given~$\F$) are distinguished optimally by a separation in~$N\!$. Likewise, $N$~is a {\em tree of the tangles in~$\vS$} if it is nested and every two maximal $\F$-tangles in~$\vS$ are distinguished optimally by a separation in~$N\!$.

\begin{definition}\label{def:robust}
Let $\vS$ be a separation system%
   \COMMENT{}
   contained in a universe~$\vU$. An ori\-entation~$\tau$ of~$S$ is {\em robust\/} in~$\vU$ if for no $s\in U$ there is a triple $\{\vr,\rv\lor\vs,\rv\lor\sv\}$ in~$\tau$%
   \COMMENT{}
   with $|\rv\lor\vs|, |\rv\lor\sv| < |r|$.%
   \footnote{If this definition appears technical, think of orientations of bipartitions of a set as designating the side they point to as~`big'. Robust orientations of set partitions are then like a finite analogue to ultrafilters: big subsets never split into two small subsets.}
\end{definition}

\noindent
   These triples will be illustrated in \cref{fig:cross}. Note that while~$\vr$ can lie in such a triple in~$\tau$ only if it lies in~$\vS$, the separation~$s$ need not lie in~$S$: it just gives rise to that triple in~$\vS$ together with~$\vr\in\vS$.

\medbreak

We shall find in every separation system~$\vS$ that lives in some universe with an injective submodular order function a tree of all the robust $\F$-tangles of~$S$, as well as a more comprehensive tree of all the robust $\F$-tangles in~$\vS$.%
   \COMMENT{}
   Here, $\F$ has to be standard and rich, as usual, but need not satisfy any further requirements. But to ease our terminology, we shall encode the robustness assumption in~$\F$, by assuming that $\F$ contains all the triples from \cref{def:robust}. We can then simply speak of (trees of) $\F$-tangles rather than of robust $\F$-tangles.%
   \COMMENT{}

The most comprehensive tree-of-tangles theorems known so far are for $\F$-tangles that are not only robust but also {\em profiles\/}: consistent orientations of a separation system~$\vS$ that contain no triple of the form ${\{\vr,\vs,\rv\lor\sv\}}$, where the supremum is once more taken in some fixed universe of separations containing~$\vS$~\cite{ProfilesNew}.

Tangles of graphs are robust profiles~\cite{DiestelBook25}.%
   \footnote{Indeed, all tangles of abstract separation systems, as defined in~\cite{AbstractTangles}, are profiles (see there).%
   \COMMENT{}
  If the ambient universe of their separation system is distributive, they are easily seen to be robust.%
   \COMMENT{}
   Graph tangles are examples of such abstract tangles in a distributive universe.}
   So the tree-of-tangles theorems for robust profiles of abstract separation systems from~\cite{ProfilesNew} generalize the tree-of-tangles theorem for graphs of Robertson and Seymour~\cite{GMX}, and our results below will further generalize those from~\cite{ProfilesNew} to $\F$-tangles that are robust but need not be profiles.\looseness=-1 %
   \COMMENT{}

There is another advance in~\cite{ProfilesNew} over~\cite{GMX}: the trees of tangles found in~\cite{ProfilesNew} are {\it canonical\/}. This means, roughly, that the automorphisms of the given graph or separation system~$\vS$ leave the tree of tangles $N\subseteq S$ invariant. This is usually proved by observing that the construction of~$N$ depends only on invariants of the graph or separation system rather than, say, of an enumeration of~$\vS$ assumed for the sake of the construction, as was the case for the first trees of tangles for graphs in~\cite{GMX}. As a consequence, such a canonical tree of tangles can then be computed by an algorithm that yields the same result regardless of the order in which the elements of the graph or separation system are presented to it.

Our tree-of-tangles theorems for all robust $\F$-tangles, \cref{thm:ToT,thm:ToTinS} below, requires that the order function given on~$S$ is injective. This assumption makes canonicity, as defined above,%
   \COMMENT{}
   trivial: unless $|S|=1$,%
   \COMMENT{}
   the identity will be the only automorphism of~$\vS$ as an {\em ordered\/} separation system.%
   \COMMENT{}
As we saw in \cref{sec:orderfunctions}, we can make a given submodular order function injective without losing its submodularity. But this conversion is not canonical:%
   \COMMENT{}
   it depends on choices not determined by the invariants of the separation system considered.%
   \footnote{In our case, this was the choice of the bijection~$\iota$ in the proof of \cref{lem:ordertoordering}.}
   If we ignore the order function on our separation system and refer in our notion of canonicity only to the unordered system, as a poset with an order-reversion involution, then the tree of tangles we shall obtain in this section will not be canonical: it will depend on the choice of (injective) order function on~$S$, because the (thoroughly ordered) tangle structure tree from which we construct it depends on it.

\medbreak

The following definition allows us to encode the robustness of  our tangles in~$\F$. Let $\vS$ be a separation system in a universe~$\vU\!$ of separations. Let
 $$\RR(\vU) := \big\{\{\vr,\rv\lor\vs,\rv\lor\sv\}\subseteq\vU :\, r,s\in U\text{ and } |\rv\lor\vs|, |\rv\lor\sv| < |r| \big\}.$$
These triples are easily identified in \cref{fig:cross}, for example.

\medbreak

   Here is our basic tree-of-tangles theorem for robust $\F$-tangles with arbitrary~$\F$, or equivalently, for all $\F$-tangles with arbitrary $\F$ that include~$\RR(\vU)$:

\begin{theorem}\label{thm:ToT}
Let $\vU\!$ be a universe of separations with an injective and structurally submodular order function. Let $k\in\R$ and $\vS:=\vUk$. Let $\F$ be any set that is rich%
   \COMMENT{}
   and standard for~$\vS\!$ and includes~$\RR(\vU)$. Let~$(T,r,\beta)$ be the unique thoroughly ordered $\F$-tangle structure tree for~$\vS$ (cf.\ \cref{thm:generalduality}). Let~$N\!$ be the set of all separations $s_v\in S$ with $v$ a tangle node of~$T$.\par
    Then $N\!$~is a tree of the tangles of~$S$; in particular, it is nested. It consists of all%
   \COMMENT{}
   the separations in~$S$ that distinguish two $\F$-tangles of~$S$ optimally.
\end{theorem}

\begin{proof}
We shall build up our proof of \cref{thm:ToT} as a sequence of a few lemmas. Let $\vU$, $k$, $\vS$, $\F$, $(T,r,\beta)$ and~$N$ be as stated in the theorem, with an injective and structurally submodular order function on~$U$.

The {\em corners\/} of two separations $r,s\in S$ are the four separations in~$U$ that have orientations of the form~$\vr\land\vs$. \cref{fig:cross} shows a corner of two bipartitions of a set, where $\vr\land\vs$ is the intersection of the subsets to which $\vr$ and~$\vs$ point.

\begin{figure}[ht]
 \center\vskip-6pt
   \includegraphics[scale=1]{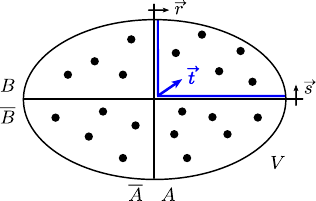}
\vskip-6pt
\caption{\small A corner $\vt = \vr\land\vs$ of two bipartitions of a set~$V\!$.}
\label{fig:cross}
\end{figure}

Two corners $t$ and~$t'$ of~$r$ and~$s$ lie {\em on the same side of~$r$\/} if these separations have orientations $\vr\ge\vt,\vtdash$. Corners on the same side of one of two crossing separations are {\em adjacent\/}; non-adjacent corners are {\em opposite\/}. Since our order function on~$U$ is structurally submodular, all four ordered pairs~$(t,t')$ of opposite corners satisfy $|t|\le|r|$ or~$|t'|\le|s|$.

In \cref{fig:cross}, the triples $\{\vr,\rv\lor\vs,\rv\lor\sv\}$ from the definition of $\RR(\vU)$ consist of $\vr$ and the two corners on the right-hand side of~$r$ both oriented towards~$r$, the last element of the triple being $\tv = \rv\lor\sv$.

\begin{lemma}\label{lem:cross}
Any two crossing separations $r,s\in S$ have two adjacent corners that either lie on the same side of~$r$ and have order~$<|r|$ or lie on the same side of~$s$ and have order~$<|s|$.
\end{lemma}

\begin{proof}
Let $|r|\le|s|$. If all four corners have order~$<|s|$ we are done; let $t_1$ be a corner of order $\ge |s|$. As $r$ and~$s$ cross we have $r\ne s\ne t_1$, so in fact $|r| < |s| < |t_1|$.%
   \COMMENT{}

By structural submodularity, the corner~$t_2$ opposite~$t_1$ has order~$\le|r|$, and hence order~$<|r|$, since $t_2\ne r$ as earlier.%
   \COMMENT{}
   If the corner~$t_3$ on the same side of~$r$ as~$t_2$ also has order~$<|r|$ we are done, so we assume not. Then $|t_3|>|r|$, and its opposite corner~$t_4$ satisfies~$|t_4|<|s|$. Since $t_4$ lies on the same side of~$s$ as~$t_2$, and $|t_2| < |r| < |s|$, this completes the proof.
\end{proof}

\begin{lemma}\label{lem:inf}
{\rm (i)} If $s\in S$ distinguishes two tangles~$\tau,\tau'$ of~$S$ optimally then $s=s_v$, where $v$ is the tangle node of~$T$ that is the infimum in~$T\!$ of the tangle leaves $\ell = \ell(\tau)$ and $\ell' = \ell(\tau')$.%
   \footnote{See \cref{thm:display} for the definition of~$\ell(\tau)$ etc.}%
   \COMMENT{}

{\rm (ii)} Conversely, if the tangle node~$v\!$ of~$T$ is the infimum in~$T\!$ of two tangle leaves $\ell = \ell(\tau)$ and~${\ell'\! = \ell(\tau')}$, then $s_v$ distinguishes~$\tau$ and~$\tau'$ optimally.
\end{lemma}

\begin{proof}
(i) By definition of~$v$, the tangles $\tau = \ucl(\beta_\ell)$ and $\tau' = \ucl(\beta_{\ell'})$ both include~$\ucl(\beta_v)$, but they differ on~$s_v$. Since~$s$ distinguishes $\tau$ and~$\tau'$ optimally, we thus have $|s| \le |s_v|$, while $s$ is not oriented by~$\ucl(\beta_v)$.%
   \COMMENT{}
   As $T$ is thoroughly ordered, this implies $|s|\ge |s_v|$,%
   \COMMENT{}
   and hence $s=s_v$ since $|\ |$ is injective.

(ii) The separation~$s_v$ distinguishes $\tau$ from~$\tau'$ by definition of~$v$. Since $T$ is thoroughly ordered, any $s\in S$ with $|s| < |s_v|$ has an orientation~$\vs$ in~$\ucl(\beta_v)$, which is a subset of both $\tau = \ucl(\beta_\ell)$ and~$\tau' = \ucl(\beta_{\ell'})$. Hence $\tau(s)=\tau'(s) = \vs$, so $s$ does not distinguish $\tau$ from~$\tau'$.
\end{proof}

\begin{lemma}\label{lem:N}
$N\!$~consists of all the separations in~$S$ that distinguish two tangles~of~$S$ optimally.
\end{lemma}

\begin{proof}
Immediate from \cref{lem:inf} and the definition of~$N$.
\end{proof}

Let us call a node~$v$ of~$T$ {\em critical\/} if $s_v$ has an orientation~$\vsv$ that is either co-trivial in~$\vS$ or such that $\ucl(\beta_v)\cup\{\vsv\}$ contains a triple from~$\RR(\vU)$.

\begin{lemma}\label{lem:notcritical}
Tangle nodes of~$T\!$%
   \COMMENT{}
   are never critical.
\end{lemma}

\begin{proof}
Suppose $T$ has a critical tangle node~$v$. Let~$\vsv$ be as in the definition of its criticality. Let $w$ be the successor of~$v$ with $\beta(vw)=\vsv$, and let $\ell\ge w$ be a tangle leaf of~$T$.%
   \COMMENT{}

Then $\vsv$ lies in the $\F$-tangle $\ucl(\beta_\ell)$. As $\F$ is standard, $\vsv$~cannot be co-trivial.
Hence $\ucl(\beta_v)\cup\{\vsv\}$ contains a triple from~$\RR(\vU)\subseteq\F$. As $\ucl(\beta_v)\cup\{\vsv\}\subseteq\ucl(\beta_\ell)$, this means that $\ell$ is not a tangle leaf, contrary to its definition.
\end{proof}

\begin{lemma}\label{lem:critical}
Let $u,v$ be non-leaf nodes of~$T\!$. If $s_u$ and~$s_v$ cross, then either $u$ or~$v$ is critical.
\end{lemma}

\begin{proof}
By \cref{lem:cross}, we can find $r\in\{s_u,s_v\}$ such that two corners~$t,t'$ of $s_u$ and~$s_v$ lie on the same side of~$r$ and have order~$<|r|$. Let $\rv$ orient~$r$ away from these corners. Then these have orientations $\vt,\vtdash\le\vr$. Note that $t,t'\in S$, since $|t|,|t'| < |r|$ and $r\in S = U_k$, as well as $\{\vr,\tv,\tvdash\}\in\RR(\vU)$.

Let $r = s_v$, say, and let $w$ be the child of~$v$ for which $\beta(vw) = \rv$. If $\rv$ is co-trivial in~$\vS$, then $v$ is critical as desired, so we assume that $\rv$ is not co-trivial. Since $T$ is an $\F$-tangle structure tree and $v$ is not a leaf, the set~$\beta_v$ has no subset in~$\F$.%
   \COMMENT{}
   As $\F$ is standard, this means that~$\beta_v$, and hence also~$\beta_w = \beta_v\cup\{\rv\}$, has no co-trival elements.

Since $t$ and~$t'$ have order~$<|r|$ but $s_v = r$, our assumption that $T$ is thoroughly ordered implies that $t$ and~$t'$ were ineligible for selection as~$s_v$: they must have orientations in~$\ucl(\beta_v)\subseteq\ucl(\beta_w)$.%
   \COMMENT{}
   Now $\ucl(\beta_w)$ is consistent by \cref{lem:closure}, because $\beta_w$ is consistent by definition of~$T$ and has no co-trivial elements.%
   \COMMENT{}
   But the only orientations of~$t$ and~$t'$ that are consistent with~$\rv\in\beta_w$ are~$\tv$ and~$\tvdash$, since $\vt,\vtdash\le\vr$. Thus, $\tv,\tvdash\in\ucl(\beta_v)$.

This makes $\{\vr,\tv,\tvdash\}\in\RR(\vU)$ a~witness of the fact that $v$ is critical.
\end{proof}

\begin{lemma}\label{lem:treeset}
$N\!$ is a tree of the tangles of~$S$. In particular, $N\!$ is nested.%
   \COMMENT{}
\end{lemma}

\begin{proof}
$N\!$ is nested by \cref{lem:notcritical,lem:critical}. By \cref{lem:N}, this makes it a tree of the tangles of~$S$.
\end{proof}

This completes our proof of \cref{thm:ToT}.\end{proof}

\medbreak

The proof of \cref{thm:ToT} starts from the premise that we have a thoroughly ordered tangle structure tree for our given separation system. \cref{thm:generalduality} provides this when $\F$ is standard and rich. Both these properties are required in the premise of the theorem; but the richness of~$\F$ is never used in the proof. However it follows from the existence of the structure tree that the proof does need, combined with the injectivity of the order function on~$S$ \cite[Theorem 4.8]{TSTs}. So we may as well require it.\looseness=-1

The tangle structure tree in which we find our tree of tangles does not, however, have to be thoroughly ordered: we could work in any tree obtained from the unique thoroughly ordered structure tree provided by \cref{thm:generalduality} by reducing it as in the proof of \cref{thm:effandirreducible} given in~\cite{TSTs}. This is because tangle nodes are never deleted in the reduction process, remain tangle nodes, and their associated separations~$s_v$ remain optimal distinguishers of any two tangles such that $v$ is the infimum in~$T$ of their tangle leaves. See~\cite[Section~5]{TSTs} for more analysis.%
   \COMMENT{}

\cref{thm:ToT} is our basic tree-of-tangles theorem for robust $\F$-tangles. It applies only to the tangles of a fixed separation system $\vS=\vUk$ inside some universe~$\vU$. The classical tree-of-tangles theorem for profiles~\cite{ProfilesNew}, however, is more comprehensive (in its more restricted context):%
   \COMMENT{}
   it finds a canonical tree of tangles for all the (maximal) tangles in~$\vS$, regardless of their order.%
   \footnote{Tangles of~$S_k$ are often called `tangles in~$\vS$ of order~$k$'.}

We can extend \cref{thm:ToT} to yield such a more general tree of tangles in~$\vS$,~too.%
   \COMMENT{}
   Our definition of tangle structure trees in this context remains the same as before, except that the notion of tangle leaves is extended to capture all the tangles in~$\vS$:

\begin{definition}\label{def:TSTinS}
An \emph{$\F$-tangle structure tree in~$\vS$} is a consistent separation tree~$T$ on~$\vS$ in which every leaf is either a tangle leaf or forbidden, and for every non-leaf~$v$ the set $\beta_v$ has no subset in $\F$.

Here, a leaf~$\ell$ of~$T$ is a {\em tangle leaf\/} if $\ucl(\beta_\ell)$ is a maximal $\F$-tangle in~$\vS$. It is {\em forbidden\/} if $\beta_\ell$ has a subset in~$\F$.
\end{definition}

\noindent
   Definitions that involve the term `tangle leaf', such as that of {\em necessary\/} edges or nodes of~$T$, adapt accordingly: they now refer to `tangle leaves' as defined above.

\medbreak

Here, then, is our more comprehensive tree-of-tangles theorem. Its tree of tangles displays more tangles than \cref{thm:ToT} does, namely, all the maximal tangles in~$\vS$ rather than just the tangles of~$S$ (which are also maximal tangles in~$\vS$). However, its premise is also stronger.%
   \COMMENT{}
   So neither theorem implies the other.

\begin{theorem}\label{thm:ToTinS}
Let $\vS$ be a universe of separations with an injective and structurally submodular order function.%
   \COMMENT{}
   Let $\F$ be any set that includes~$\RR(\vS)$ and is rich and standard for~$\vSk$ for every $k\in\R$.%
   \COMMENT{}
   Then there exists a tree of the tangles in~$\vS$.
\end{theorem}

\begin{proof}
We start by observing that the thoroughly ordered $\F$-tangle structure trees $(T_k,r_0,\beta^k)$ provided by \cref{thm:ToT} for the separation systems~$\vSk\subseteq\vS$ are unique by \cref{thm:generalduality}. And they are naturally nested: growing from the same root~$r_0$, they satisfy $T_i\subseteq T_j$ and $\beta^i\subseteq\beta^j$ for all~$i<j$.%
   \COMMENT{}

Let $n$ be large enough that $S=S_n$, and let $T:=T_n$. Then every non-leaf node~$v$ of~$T$ is a leaf of some~$T_k$: just choose $k$ big enough that $v\in T_k$, but small enough that $|s_v|\ge k$. As $v$ is a non-leaf node of~$T$, the fact that $(T_n,r_0,\beta^n)$ is an $\F$-tangle structure tree implies%
   \COMMENT{}
   that $\beta^n_v\supseteq\beta^k_v$ has no subset in~$\F$. Thus, $v$~is a tangle leaf of~$T_k$ in the usual sense of \cref{sec:TSTs}.%
   \COMMENT{}%
   \looseness=-1

Let $T'$ be obtained from~$T$ by deleting any pairs $\ell,\ell'$ of forbidden leaves of~$T$ that are children of the same node, and let $\beta:= {\beta^n\!\restriction\! E(T')}$. Then every leaf of~$T'$ is either a forbidden leaf of~$T$, or a tangle leaf~$\ell$ of~$T_k$ for some~$k$.%
  \COMMENT{}
   In the latter case, its associated tangle~$\tau$ of~$S_k$ is a maximal tangle in~$\vS$: it cannot extend to a tangle $\tau'\supsetneq\tau$ of any~$S_{k'}$. Indeed, this~$\tau'$ would be associated with a tangle leaf $\ell'>_r \ell$ of~$T_{k'}$, so $\tau'$  would include either $\beta_\ell\cup\{\vsl\}$ or~$\beta_\ell\cup\{\svl\}$. This cannot happen for any $\F$-tangle~$\tau'$, as both these sets have a subset in~$\F$ by the choice of~$\ell$.%
   \COMMENT{}

Thus, every leaf $\ell$ of~$T'\!$ is either forbidden (and also a leaf of~$T$), or it is a tangle leaf in our new sense that $\ucl(\beta_\ell)$ is a maximal $\F$-tangle in~$\vS$. Conversely, for every maximal tangle~$\tau$ in~$\vS$ there is a leaf $\ell$ of~$T'$ such that $\ucl(\beta_\ell) = \tau$. But, unlike in~$T$, for every non-leaf node~$v$ of~$T'$ there exists a tangle leaf $\ell > v$ of~$T'\!$.%
   \COMMENT{}

Let~$V\!$ be the set of all non-leaf nodes of~$T'\!$ that are not the parent of a forbidden leaf of~$T$.%
   \COMMENT{}
   This~$V\!$ is precisely the set of all tangle nodes of~$T'\!$.%
   \COMMENT{}
   Let $N = \{\,s_v\mid v\in V\}$. We shall prove that $N$ is our desired tree of tangles: that it is nested, and that it consists of precisely the separations in~$S$ that distinguish two maximal $\F$-tangles in~$\vS$ optimally.

Let us show first that every two maximal tangles in~$\vS$, say a~tangle $\tau_i$ of~$S_i$ and a tangle~$\tau_j$ of~$S_j$ with $i\le j$, are distinguished optimally by some separation in~$N$. Let us apply \cref{thm:ToT} with $k:=i$ to the tangles~$\tau_i$ and $\tau_j\cap\vSi\ne\tau_i$ of~$S_i$.%
   \COMMENT{}
   By the proof of the theorem,%
   \COMMENT{}
   these tangles are distinguished optimally by~$s_v$ for $v$ the infimum of their tangle leaves in~$T_i$. This~$v$ lies in~$V\!$, and $s_v$ also distinguishes $\tau_i$ and~$\tau_j$ optimally as tangles in~$\vS$, i.e., as partial orientations of~$S$.\looseness=-1

Conversely, let us show that every $s_v\in N$ distinguishes some such pair~$\tau_i,\tau_j$ of tangles in~$\vS$ optimally. As $v$ lies in~$V\!$, it has two children in~$T'\!$, neither of which is a forbidden leaf of~$T$. Since all forbidden leaves of~$T'$ are also forbidden leaves of~$T$, the definition of~$T'\!$ implies that $v$ is the infimum of some tangle leaves $\ell_i$ and~$\ell_j$ of~$T'\!$%
   \COMMENT{}
   associated with tangles $\tau_i$ of~$S_i$ and $\tau_j$ of~$S_j$. By our earlier arguments, these are maximal tangles in~$\vS$, and $s_v$~distinguishes them optimally.

To see that $N$~is nested, consider distinct nodes $u,v\in V\!$ in~$T$.%
   \COMMENT{}
   Suppose $s_u$ and~$s_v$ cross. By \cref{lem:cross}, we can find $r\in\{s_u,s_v\}$ such that two corners~$t,t'$ of $s_u$ and~$s_v$ lie on the same side of~$r$ and have order~$<|r|$. Let us assume that $r=s_v$.
   As earlier in our proof,%
   \COMMENT{}
   $v$~is a tangle node in some~$T_i$. But now $v$ is critical in~$T_i$, by the proof%
   \COMMENT{}
   of \cref{lem:critical}. (Note that this proof no longer considers~$u$ after identifying~$v$ by \cref{lem:cross}, exactly as we did here. This is important, since in our context $u$~may not be a node of~$T_i$.) This contradicts \cref{lem:notcritical} applied to~$T_i$.%
   \COMMENT{}%
   \footnote{\lineskiplimit=-3pt We cannot apply the two lemmas to~$T'\!$ directly because~$T'\!$, unlike $T$ and~$T_i$, need not be a tangle structure tree of any~$\vSdash\!$ with $S'\subseteq S$.}%
   \COMMENT{}
   \end{proof}

Let us summarize the details from our proof of \cref{thm:ToTinS}:

\begin{corollary}\label{lem:ToTinS}
Let $\vS$ be a universe of separations with an injective and structurally submodular order function. Let $\F$ be any set that includes~$\RR(\vS)$ and is rich and standard for~$\vSk$ for every $k\in\R$. Let $T'\!$ be obtained from the unique%
   \COMMENT{}
   thoroughly ordered $\F$-tangle structure tree $(T,r,\beta)$ for~$\vS$ by deleting any pairs $\ell,\ell'$ of forbidden leaves of~$T\!$ that are children of the same node. Let~$V\!$ be the set of tangle nodes of~$T'\!$.\par%
   \COMMENT{}
    Then $N = \{\,s_v\mid v\in V\}$ is a tree of the tangles in~$\vS$. It is nested and consists of all the separations in~$S$ that distinguish two maximal $\F$-tangles in~$\vS$ optimally.
\end{corollary}

\medskip\section{Zooming in: how to make consistent subsets more efficient}\label{sec:measure}


\noindent
   When we compared our new tangle-tree duality results from \cref{sec:TTD} with the classical such theorems, such as \cref{thm:prevduality}~\cite{TangleTreeAbstract}, we observed two things. One was that our dichotomy \cref{thm:generalduality2} is more general than that from~\cite{TangleTreeAbstract} since, unlike there, our sets~$\F$ of forbidden subsets need not be stars of separations. Due to this greater generality, however, our witnesses for the non-existence of tangles, though `tree-like' in the sense that they are $\F$-trees, do not impose a tree structure on whatever our separation system separates (e.g., a graph or dataset), as \cref{thm:prevduality} from~\cite{TangleTreeAbstract} does.

For the case that $\F$ does consist of stars we went on to prove that our results imply a tangle-tree duality theorem just like \cref{thm:prevduality}, albeit from a different premise. At the end of \cref{sec:TTD} we remarked that, since our premise (that~$\F$ is rich) differs so fundamentally from the premise in~\cite{TangleTreeAbstract} (that $\vS$ is $\F$-separable), it is unlikely that our results could imply those of~\cite{TangleTreeAbstract} formally.

Our aim in this section is to show that our premise (and hence, tangle-tree duality) does, however, follow from a slight strengthening of $\F$-separability, one that is actually verified and used when the duality theorem of~\cite{TangleTreeAbstract} is applied~\cite{TangleTreeGraphsMatroids}. As a result, we shall be able to deduce all the applications of~\cite{TangleTreeAbstract} to concrete structures of interest~\cite{TangleTreeGraphsMatroids} from our results. This will be done in~\cite{TSTapps}.

Assuming a structurally submodular injective order function on~$\vS$, we shall prove that $\F$ is rich as soon as it is {\em closed under shifting\/}. We can then apply our \cref{cor:E} to obtain the same tangle-tree dichotomy as offered by \cref{thm:prevduality}, which also makes this assumption about~$\F$. In fact, our notion of `closed under shifting' will be weaker than that in~\cite{TangleTreeAbstract},%
   \COMMENT{}
   making our result stronger. While \cite{TangleTreeAbstract} makes no formal submodularity assumption, it assumes a consequence of it (separability of~$\vS$), which in practice is always established by assuming submodularity~\cite[Lemma~3.4]{TangleTreeGraphsMatroids}.%
   \COMMENT{}

Our strategy for proving, from assumptions similar to those made in~\cite{TangleTreeAbstract}, that some given~$\F$ is rich, is as follows. Recall that $\F$ is {\em rich\/} if every consistent orientation~$\tau$ of~$S$ that has a subset in~$\F$ also has an efficient such subset: one that is `maximally focused' in the sense that none of its elements is eclipsed by another separation in~$\tau$. Our strategy will be to start from an arbitrary set $\sigma\in\F$, and then to make it more focused by iteratively `zooming in' until it is efficient.

For intuition, think of any consistent orientation~$\tau$ of~$\vS$ as pointing towards some place in a dataset~$V\!$ that $S$ may be separating. If $\tau$ is a tangle, this could be a cluster; if not, it could be the (possibly empty) interior $\bigcap\sigma$ of a star~$\sigma\in\F$. For two oriented separations $\vr,\vs$ in~$\tau$, think of $\vs > \vr$ as expressing that $\vs$ points towards~$r$ and, beyond it, to whatever place $\vr$ points to in~$V\!$. Thus, while $\vr$ and~$\vs$ point to the same place, $\vr$~is closer to it. Hence if $\vr$ eclipses~$\vs$, then replacing~$\vs$ with~$\vr$ in some subset~$\sigma$ of~$\tau$~-- e.g., one in~$\F$, or a set of elements whose closure is a tangle of~$S$~-- makes~$\sigma$ more focused on the place to which~$\tau$ points.

Our strategy for making a given set $\sigma\in\F$ more efficient will thus be to replace, iteratively, an element of~$\sigma$ with another element of~$\tau$ that eclipses it. The rest of~$\sigma$ will either be left unchanged (at the start of this section) or `shifted' to one side of the eclipsing separation, as in~\cite{TangleTreeAbstract} (later in this section).

Let us call a collection $\F$ of subsets of~$\vS$ {\em closed under eclipsing\/} in $\tau\subseteq\vS$ if any set~$\sigma'$ obtained from a subset $\sigma\in\F$ of~$\tau$ by replacing some $\vs\in\sigma$ with a separation $\vr\in\tau$ that weakly eclipses~$\vs$ will also  lie in~$\F$.%
   \COMMENT{}
   Note that $\sigma'$ will be strictly smaller than~$\sigma$ in the natural extension of our partial ordering~$\le$ on~$\vS$ to~$2^\vS\!$, in which $\sigma\ge\sigma'$ means that there exists a map $f\colon \sigma\to\sigma'$ such that $\vs\ge f(\vs)$ for all $\vs\in\sigma$.

\begin{lemma}\label{lem:psiefficient}
If $\vS$ is an ordered%
   \COMMENT{}
   separation system and $\F\subseteq 2^\vS\!$ is closed under eclipsing in every consistent orientation of~$S$, then $\F$ is rich for~${\vS}$.
\end{lemma}

\begin{proof}
To prove that~$\F$ is rich, let $\tau$ be any consistent orientation of~$S$ that has a subset in~$\F$. We claim that every set~$\sigma$ that is minimal in~$\F\cap 2^\tau\ne\emptyset$ with respect to our partial ordering on~$2^\vS\!$ is strongly efficient in~$\tau$.

Indeed, suppose $\sigma$ is not strongly efficient in~$\tau$. Then there exists an $\vr\in\tau$%
   \COMMENT{}
   that weakly eclipses some $\vs\in\sigma$. Let $\sigma' := (\sigma\setminus\{\vs\})\cup\{\vr\}$. Then $\sigma > \sigma'\in\F$, since $\F$ is closed under eclipsing in~$\tau$. This contradicts the minimality of~$\sigma$, since $\sigma'\subseteq\tau$.%
   \COMMENT{}
\end{proof}

We now introduce our second way of making a set $\sigma\in\F$ more focused, which is essentially the shifting operation from \cite{TangleTreeGraphsMatroids,TangleTreeAbstract}.%
   \COMMENT{}
   Let $\vs\in\vS$ be non-trivial in~$\vS$ and non-degenerate. Write~$S_{\le\vs}$ for the set of all $t\in S$ to which $\vs$ points, those with an orientation~$\vt\le\vs$. Given any $\vr\le\vs$ in~$\vS$, define the $\vS_{\le\vs}\to\vU$ map%
   \footnote{Note that, as $\vs$ is non-trivial, $\vt\le\vs$ and $\tv\le\vs$ cannot both hold unless $s=t$, in which case $\vs\land\vr = \vr$ (rather than its inverse) is chosen as the image of~$\vs$, and $\rv$ as the image of~$\sv$.}
\[f\!\downarrow_{\vr}^{\vs}(\vt):=
\begin{cases}
\vt\wedge\vr &\text{if } \sv\ne\vt\le\vs;\\
(\tv\wedge\vr)^* & \text{if } \sv\ne\tv\le\vs.
\end{cases}\]
   It is easy to see that such {\em shifting maps\/} preserve the partial ordering between their arguments; in particular, they map stars to stars~\cite{TangleTreeAbstract}.
We say that a separation $\vr\le\vs$ in~$\vS$ {\em emulates\/}~$\vs$ in~$\vS$ if our shifting map has its image inside~$\vS$, i.e., if $\vt\land\vr\in\vS$ for every $\vt\in\vS$ with $\sv\ne\vt\le\vs$.

Let us show that shifting stars can reduce their {\em order\/}, the sum of the orders of their elements under a given order function on our separation system.

\begin{lemma}\label{lem:shift}
Let $\vU$ be a universe of separations with a structurally submodular order function, and let $S = U_\ell$ for some~$\ell\in\R$.%
   \COMMENT{}
   Let $\tau$ be a consistent orientation of~$S$. Assume that some element~$\vs$ of a star $\sigma\subseteq\tau$ is non-trivial in~$\vS$ and eclipsed%
   \COMMENT{}
   by some $\vr\in\tau$. Then $\vr$ can be chosen so that it emulates~$\vs$ in~$\vS$ and $\sigma':= f\!\downarrow_{\vr}^{\vs}(\sigma)$%
   \COMMENT{}
   is a star in~$\tau$ of lower order than~$\sigma$.
\end{lemma}

\begin{proof}
Choose~$\vr$ with $|r|$ minimum, and subject to this maximal under the partial ordering on~$\vS$. Let us show first that $\vr$ emulates~$\vs$ in~$\vS$.%
   \COMMENT{}

We have to show that $\vr\land\vt\in\vS$ for any $\vt\in\vS$ with $\sv\ne\vt\le\vs$, i.e., that $|\vr\land\vt| < \ell$. By the structural submodularity of our order function we have $|\vr\land\vt|\le |t| < \ell$ as desired if $|\vr\lor\vt| > |r|$,%
   \COMMENT{}
   so let us assume that $|\vr\lor\vt|\le |r|$. Then $\vr\lor\vt\in\vS$;%
   \COMMENT{}
   let us show that, in fact, $\vr\lor\vt\in\tau$.%
   \COMMENT{}

 As $\vr\lor\vt\ge\vr\in\tau$, the consistency of~$\tau$ implies $\vr\lor\vt\in\tau$ as desired, unless the unoriented separation underlying $\vr\lor\vt$ is~$r$. If $\vr\lor\vt = \vr$, we are done since $\vr\in\tau$, so let us assume that $\rv = \vr\lor\vt\ge \vr$. Then $\vs\ge\vr\lor\vt = \rv$, since $\vs\ge\vr$ as well as $\vs\ge\vt$ by assumption, which makes $\vs$ trivial witnessed by~$r$, contrary to assumption, unless $s=r$. But if $\vs=\vr$ we have $\vr\lor\vt = \vs\lor\vt = \vs\in\tau$, while if $\vs = \rv$ we have $\vr\lor\vt = \rv = \vs\in\tau$. This completes our proof that $\vr\lor\vt\in\tau$.

As noted, we have $\vs\ge\vr\lor\vt$,%
   \COMMENT{}
   as well as $|\vr\lor\vt|\le |r| < |s|$, since $\vr$ eclipses~$\vs$. So $\vr\lor\vt$ eclipses~$\vs$,%
   \COMMENT{}
   and was thus a candidate for~$\vr$. As $|\vr\lor\vt|\le |r|$, this contradicts the choice of~$\vr$ unless $\vr\lor\vt = \vr$. But in that case $\vr\ge\vt$,%
   \COMMENT{}
   giving $\vr\land\vt = \vt\in\vS$ as desired.%
   \COMMENT{}
   This completes our proof that $\vr$ emulates~$\vs$.

As shifting preserves the partial ordering on~$\vS$, we know that $\sigma'$ is a star.%
   \COMMENT{}
   Let us show that it has lower order than~$\sigma$. Since $f\!\downarrow_{\vr}^{\vs}(\vs) = \vr$ and $|r| < |s|$,%
   \COMMENT{}
   it suffices to show that $\big|f\!\downarrow_{\vr}^{\vs}(\vt)\big| = |\vr\land\vt|\le |t|$ for any other elements~$\tv$ of~$\sigma$.%
   \COMMENT{}
   But this follows from structural submodularity and the choice of~$\vr$, as earlier, unless $\vr\land\vt = \vt$.%
   \COMMENT{}
   But in that case we  also have $|\vr\land\vt| = |t|$ as desired.

It remains to show that $\sigma'\subseteq\tau$. We have $f\!\downarrow_{\vr}^{\vs}(\vs) = \vr\in\tau$ by explicit assumption about~$\vr$. For any $\tv\in\sigma$ other than~$\vs$, we have $\vs\ge\vt$, since $\sigma$ is a star, and hence $f\!\downarrow_{\vr}^{\vs}(\tv) = (\vt\land\vr)^* = \tv\lor\rv \ge\tv\in\tau$ by the definition of~$f\!\downarrow_{\vr}^{\vs}$. 
This implies $f\!\downarrow_{\vr}^{\vs}(\tv)\in\tau$ by the consistency of~$\tau$ unless $f\!\downarrow_{\vr}^{\vs}(\tv) = \vt$,%
   \COMMENT{}
   so let us show that this cannot be the case.

If it is, then we have shown that $\vt = f\!\downarrow_{\vr}^{\vs}(\tv)\ge\tv$. Then $\vs\ge\vt\ge\tv$, which contradicts our assumption that $\vs$ is non-trivial unless $s=t$. But in that case we have $\vs=\vt$, since $\vs\ne\tv$ as these are distinct elements of~$\sigma$. But if $\vs = \vt$ then $\rv = f\!\downarrow_{\vr}^{\vs}(\sv) = f\!\downarrow_{\vr}^{\vs}(\tv) = \vt = \vs$, which contradicts the fact that $|r| < |s|$.
\end{proof}

Call a set $\F\subseteq 2^\vS\!$ of stars \emph{closed under shifting}%
   \footnote{Our definition of `closed under shifting' is weaker than that in~\cite{TangleTreeGraphsMatroids,TangleTreeAbstract}, making our subsequent results stronger. We can even strengthen them further: as is easily checked, we can weaken in both \cref{lem:psiefficienter,thm:newduality} the assumption that $\F$ is closed under shifting in\,$\vS$ to assuming that $\F$ is closed under shifting in every consistent orientation of~$S$. We omitted this for simplicity.}\COMMENT{}
   in~$\tau\subseteq\vS$ if for every $\sigma\in\F\cap 2^\tau$ and every $\vr\in\tau$ that weakly eclipses%
   \COMMENT{}
   and emulates some non-trivial $\vs\in\sigma$%
   \COMMENT{}
   in~$\vS$ we have~$f\!\downarrow_{\vr}^{\vs}(\sigma)\in\F$.

\begin{lemma}\label{lem:psiefficienter}
Let $\vU$ be a universe of separations with a structurally submodular injective order function, and let $S = U_\ell$ for some~$\ell\in\R$. Let $\F$ be a set of stars of non-trivial%
    \COMMENT{}
    separations in~$\vS$ that is closed under shifting in~$\vS$.$^{\thefootnote}$ Then $\F$ is rich for~$\vS$.\looseness=-1
\end{lemma}

\begin{proof}
To prove that~$\F$ is rich, consider any consistent orientation~$\tau$ of~$S$ that includes some~$\sigma\in\F$. Choose such a star~$\sigma$ in~$\tau$ of minimum order. We shall show that $\sigma$ is strongly efficient in~$\tau$.

If not, there exists an $\vr \in \tau$ that weakly eclipses some $\vs\in\sigma$. Since our order function is injective, any such $\vr$~in fact eclipses~$\vs$. By \cref{lem:shift}, we can choose~$\vr$ so that it emulates~$\vs$ in~$\vS$ and $\sigma':= f\!\downarrow_{\vr}^{\vs}(\sigma)$%
   \COMMENT{}
   is a star in~$\tau$ of lower order than~$\sigma$. As $\F$ is closed under shifting we have $\sigma'\in\F$, which contradicts our choice of~$\sigma$.
\end{proof}

\begin{theorem}\label{thm:newduality}
Let $\vU\!$ be a universe of separations with either a submodular or an injective structurally submodular order function. Let $S = U_\ell$ for some~$\ell\in\R$. Assume that $\vS$ has no trivial elements. Let $\F$ be a set of stars of separations in~$\vS$ that is closed under shifting in~$\vS$.$^{\thefootnote}$ Then exactly one of the following assertions holds:\looseness=-1
\begin{enumerate}\itemsep2pt\vskip2pt
  \item there exists an $\F$-tangle of~$S$;
  \item there exists an $S$-tree over~$\F$.
\end{enumerate}
\end{theorem}

\begin{proof}
Recall from the proof of \cref{lem:ordertanglestoorderingtangles} that if $S$ has the form of $S=U_\ell$ with respect to a given order function~$o$, it also has this form with respect to any refinement of~$o$. By \cref{lem:newordertoordering} we may therefore assume that our order function on~$U$ is injective and structurally submodular. By \cref{lem:psiefficienter}, $\F$~is rich for~$\vS$. The~theorem now follows from \cref{cor:E}.
\end{proof}

While \cref{thm:newduality} is probably to be the most widely applicable tangle-tree duality theorem to date, it may appear restrictive that $S$ must have the form of~$U_\ell$, and that we need a submodular order function on~$U\!$. The original tangle-tree duality theorem from~\cite{TangleTreeAbstract}, \cref{thm:prevduality}, makes no such assumptions. However, it makes another assumption, one not made in \cref{thm:newduality}: that $\vS$ is $\F$-separable.

This is not a coincidence: $\F$-separability is essentially a technical summary of exactly those consequences of our two more natural assumptions, submodularity and ${S=U_\ell}$, that are technically needed in the proof of \cref{thm:prevduality}.

Finally, the premise of \cref{thm:newduality} differs from that of \cref{thm:prevduality} in that it requires~$\vS$ not to have trivial elements.%
   \COMMENT{}
   We need this for our conversion of $\F$-trees into $S$-trees over~$\F$ in \cref{thm:main}. But it entails no loss of generality: the trivial separations of any separation system lie in all consistent orientations, so we can just delete them from our given system and add them back later if desired~\cite{ASS,TreeSets}.

\bibliographystyle{alpha}
\bibliography{collective}

\end{document}